\documentclass[12pt, a4paper]{amsart}

\usepackage{verbatim}
\usepackage{amssymb}
\usepackage{textcomp}
\usepackage{mathrsfs}
\usepackage{pifont}

\usepackage[all]{xy}
\CompileMatrices

\theoremstyle{plain}

\newtheorem{theorem}{Theorem}[section]
\newtheorem{lemma}[theorem]{Lemma}
\newtheorem{fact}[theorem]{Fact}
\newtheorem{corollary}[theorem]{Corollary}
\newtheorem{proposition}[theorem]{Proposition}
\newtheorem{conjecture}[theorem]{Conjecture}

\theoremstyle{definition}
\newtheorem{definition}[theorem]{Definition}
\newtheorem{example}[theorem]{Example}

\newtheorem{question}[theorem]{Question}

\newtheorem{remark}[theorem]{Remark}

\theoremstyle{remark}
\newtheorem{clm}{Claim}
\newtheorem*{claim}{Claim}

\numberwithin{figure}{section}

\DeclareMathOperator{\im}{Im}
\DeclareMathOperator{\aut}{Aut}
\DeclareMathOperator{\SL}{SL}
\DeclareMathOperator{\GL}{GL}

\DeclareMathOperator{\SO}{SO}
\DeclareMathOperator{\Sp}{Sp}
\DeclareMathOperator{\St}{{St}_{2n}^{\rm{sym}}}

\DeclareMathOperator{\dcl}{dcl}

\DeclareMathOperator{\Hom}{Hom}

\DeclareMathOperator{\Stab}{Stab}

\DeclareMathOperator{\K2}{K_{2}^{\rm{sym}}}

\DeclareMathOperator{\Th}{Th}

\DeclareMathOperator{\chr}{char}

\DeclareMathOperator{\cw}{cw}

\DeclareMathOperator{\Tr}{Tr}

\newcommand{\G}{{\mathcal G}}
\renewcommand{\H}{{\mathcal H}}
\newcommand{\U}{{\mathcal U}}

\newcommand{\N}{{\mathbb N}}
\newcommand{\Z}{{\mathbb Z}}
\newcommand{\R}{{\mathbb R}}
\newcommand{\Q}{{\mathbb Q}}
\newcommand{\id}{{\mathrm{id}}}

\newcommand{\Kcal}{{\mathcal K}}

\newcommand{\F}{{\mathbb F}}

\begin{document}

\title{On model-theoretic connected components in some group extensions}
\date{\today}

\author{Jakub Gismatullin \and Krzysztof Krupi\'nski}
\thanks{The first author is supported by the Marie Curie Intra-European Fellowship MODGROUP no. PIEF-GA-2009-254123 and European Social Fund under project POKL.04.01.01-00-054/10-00. Both authors are supported by Polish Government grant N N201 545938 and NCN grant DEC-2012/07/B/ST1/03513}


\address{Instytut Matematyczny Uniwersytetu Wroc{\l}awskiego, pl. Grunwaldzki 2/4, 50-384 Wroc{\l}aw, Poland and School of Mathematics, University of Leeds, Woodhouse Lane, Leeds, LS2 9JT, UK}
\email{gismat@math.uni.wroc.pl, www.math.uni.wroc.pl/\~{}gismat}
\address{Instytut Matematyczny Uniwersytetu Wroc{\l}awskiego, pl. Grunwaldzki 2/4, 50-384 Wroc{\l}aw, Poland}
\email{kkrup@math.uni.wroc.pl, www.math.uni.wroc.pl/\~{}kkrup}


\begin{abstract} We analyze model-theoretic connected components in extensions of a given group by abelian groups which are defined by means of 2-cocycles with finite image. We characterize, in terms of these 2-cocycles,  when the smallest type-definable subgroup of the corresponding extension differs from the smallest invariant subgroup. In some situations, we also describe the quotient of these two connected components.

Using our general results about extensions of groups together with Matsumoto-Moore theory or various quasi-characters considered in bounded cohomology, we obtain new classes of examples of groups whose smallest type-definable subgroup of bounded index differs from the smallest invariant subgroup of bounded index. This includes the first known example of a group with this property found by Conversano and Pillay, namely the universal cover of $\SL_2(\R)$ (intepreted in a monster model), as well as various examples of different nature, e.g. some central extensions of free groups or of fundamental groups of closed orientable surfaces. We also obtain a variant of the example of Conversano and Pillay for $\SL_2(\Z)$ instead of $\SL_2(\R)$, which (as most of our examples) was not accessible by the previously known methods.
\end{abstract}

\maketitle

\tableofcontents



\section{Introduction}

Assume $G$ is a group $\emptyset$-definable in a monster model of some first order theory, and let $B$ be a (small) set of parameters from this model. The following connected components play a very important role in the study of groups from the model-theoretic perspective: 

\begin{itemize}
\item the intersection of all $B$-definable subgroups of $G$ of finite index, denoted by $G^0_B$,
\item the smallest $B$-type-definable subgroup of $G$ of bounded index, denoted by $G^{00}_B$,
\item the smallest $B$-invariant subgroup of $G$ of bounded index, denoted by $G^{000}_B$ or by $G^{\infty}_B$.
\end{itemize}

It is clear that $G^{000}_B \leq G^{00}_B \leq G^{0}_B \leq G$, and it is easy to show that all these groups are normal in $G$. Sometimes these connected components do not depend on the choice of $B$, e.g. in NIP theories. In such situations, we skip the parameter set $B$, and we say that the appropriate connected component exists, e.g. we write $G^{000}$ and say that $G^{000}$ exists (notice that then $G^{000}$ is the smallest invariant, over an arbitrarily chosen small set of parameters, subgroup of bounded index).

The significance of the above connected components was discussed in various papers (e.g. see \cite{gis, modcon}).
So, we will only briefly discuss the motivation to consider them. Originally, $G^0$ played a fundamental role in the study of generic types in stable groups, for example due to the fact that $G/G^0$ is always a profinite group. The importance of $G^{00}$ stems from the fact that it allows to associate with $G$ a compact topological group $G/G^{00}$ (with the so-called logic topology). This becomes particularly interesting in $o$-minimal structures due to Pillay's conjecture which describes $G/G^{00}$ as a compact Lie group of an appropriate dimension, and so associates with the group $G$ a classical mathematical object $G/G^{00}$ \cite{peter}. A common motivation to consider all three connected components also comes from their relationships with strong types in various senses (strong types, Kim-Pillay types and Lascar strong types), which in turn are essential notions in the study of stable, simple and NIP theories. For details on these relationships see \cite[Section 3]{gis}. Notice  also that all the quotients of the connected components, e.g. $G/G^0$, $G/G^{00}$, $G/G^{000}$ or $G^{00}/G^{000}$, are certain natural model-theoretic invariants of the group $G$ (in the sense that they do not depend on the choice of the monster model), so it is a natural model-theoretic task to understand them.

Our main goal is to find examples, or, more desirably, general methods of constructing them, of groups $G$ for which $G^{00}_B \ne G^{000}_B$. This problem arose during the work on the first author's Ph.D. thesis under the supervision of L. Newelski, 
and it appears 
in \cite[Page 496]{gis}. Recall that any such example leads to a new 
example of a non-$G$-compact theory (see \cite[Corollary 3.6]{gis}). 

In \cite{conv_pillay}, the authors found first (strongly related to each other) examples of groups $G$ for which $G^{00} \ne G^{000}$. Their example is a monster model of the topological universal cover $\widetilde{\SL_2(\R)}$ of $\SL_2(\R)$.
In particular, $\widetilde{\SL_2(\R)}$ is a central extension of $\SL_2(\R)$ by $\Z$ given by a definable (in the two sorted structure $((\Z,+),(\R,+,\cdot))$) 2-cocycle $h \colon \SL_2(\R) \times \SL_2(\R) \to \Z$ with finite image.


This led us to the following general question. 

\begin{question}
When does an extension $\widetilde{G}$ of a group $G$ by an abelian group $A$ satisfy $\widetilde{G}^{00}_B \ne \widetilde{G}^{000}_B$ for some parameter set $B$ (working in a monster model)? 
\end{question}


We would like to emphasis that we consider this problem in a general algebraic context, i.e., without assuming that $\widetilde{G}$ is a topological universal cover of a topological group or that $G$ is definable in an $o$-minimal structure. The only restriction that we make is the assumption that the 2-cocycle $h \colon G \times G \to A$ defining our extension is definable (in the model in which we are working)  and has finite image. 


When we talk about connected components, we always compute them in a monster model.
In Section \ref{sec:main}, we characterize when $\widetilde{G}^{00}_B \ne \widetilde{G}^{000}_B$ in terms of the underlying 2-cocycle $h \colon G \times G \to A$, and we analyze the quotient  $\widetilde{G}^{00}_B / \widetilde{G}^{000}_B$. We obtain two corollaries of our characterization, which are used in Section \ref{sec:central} to find new classes of examples where these two connected components are different. 
The first class of examples, generalizing the example of Conversano and Pillay to some central extensions of $\Sp_{2n}(k)$ (and of some other related groups) for $k$ being any ordered field, is obtained by application of symplectic Steinberg symbols and Matsumoto-Moore theory \cite{rehm, stein}. The second family of examples, including some central extensions of free groups or of fundamental groups of closed orientable surfaces, are obtained by application of various quasi-characters considered in bounded cohomology. 
As a corollary, we get that in a non-abelian free group expanded by predicates for all subsets the two connected components differ. Using our main characterization, we also obtain that in the central extension of $\SL_2(\Z)$ by $\Z$ defined by the 2-cocycle defining $\widetilde{\SL_2(\R)}$ restricted to $\SL_2(\Z) \times \SL_2(\Z)$, the two connected components differ. Section \ref{sec:finite} contains some auxiliary results which allow to produce yet more examples from the existing ones (concrete applications are described in Section \ref{sec:central}).

The relevant definitions and facts concerning extensions of groups and 2-cocycles are given in the initial parts of Sections \ref{sec:main} and \ref{sec:central}. Here we only recall a few definitions and facts from model theory.

Suppose a set $X$ is type-definable in a monster model. Let $E$ be a bounded (i.e., with boundedly many classes), type-definable equivalence relation on $X$. The \emph{logic topology} on $X/E$ is a topology whose closed sets are the sets with type-definable preimages by the quotient map. It turns out that this is always a compact (Hausdorff) topology. In particular, the equivalence relation on $G$ of lying in the same coset modulo $G^{00}_B$ is always bounded and type-definable, so we have the logic topology on $G/G^{00}_B$. With this topology, $G/G^{00}_B$ becomes a compact, topological group. For the basic properties of the logic topology see for example \cite[Section 2]{pill} and \cite[Proposition 3.5 1.]{gis}.

Now, we recall the notion of a thick subset of a group.

\begin{definition}
A subset $P$ of an arbitrary group $G$ (not necessarily
sufficiently saturated) is called \emph{$n$-thick}, where $n \in \N$, if it is symmetric (that is, $P = P^{-1}$)
and for every  $g_1, \dots , g_n$ from $G$ there are $1 \leq i < j \leq n$ such that $g_i^{-1}
g_j \in P$. We
say that $P$ is \emph{thick} if $P$ is $n$-thick for some $n \in \N$. 
\end{definition}

In  the above definition, instead of  $g_i^{-1}
g_j \in P$ we can write $g_ig_j^{-1}
\in  P$, as one can always consider elements $g_1^{-1}
, \dots , g_n^{-1}$ instead of $g_1, \dots , g_n$. Recall that the intersection of any two thick subsets of a group is always thick, and the image of a thick subset of a group by an epimorphism is also thick.

\begin{fact}[{\cite[Lemma 1.3, Lemma 1.5]{abscon}}]\label{fact:new 1}
Let $G$ be a group $\emptyset$-definable in a monster model, and $B$ a small set of parameters.
\begin{enumerate}
\item[(i)] The component  $G^{000}_B$ is generated by the intersection of all $B$-definable thick subsets
of G.
\item[(ii)] The component $G^{00}_B$
can be written as the intersection of some family ${\mathcal P}$ of $B$-definable thick subsets of $G$, where ${\mathcal P}$ is closed under finite intersections.
\end{enumerate}
\end{fact}

We will sometimes  use the notion of an absolutely connected group from \cite{abscon}. 


\begin{definition}[{\cite[Definition 2.3, Propositions 2.1 and 2.2]{abscon}}] \label{def:abscon}
Let $G$ be an infinite group.
\begin{enumerate}
\item Suppose $G$ is given with some first order structure $\G = (G,\cdot,\ldots)$. We say that $G$ is \emph{definably absolutely connected} if ${G^*}^{000}$ exists and $G^* = {G^*}^{000}$, where  $G^*$ is the interpretation of $G$ in a monster model $\G^* \succ \G$. 
\item We say that $G$ is \emph{absolutely connected} if for an arbitrary first order expansion $\G = (G,\cdot,\ldots)$ of $G$ the component ${G^*}^{000}$ exists and $G^* = {G^*}^{000}$, where $G^*$ is an interpretation of $G$ in a monster model $\G^* \succ \G$.
\end{enumerate}
\end{definition}


The main examples of absolutely connected groups are simply connected Chevalley groups \cite[Theorems 6.5, 6.7]{abscon}. That is, if $k$ is an arbitrary infinite field and $G$ is a $k$-split, semisimple, simply connected linear algebraic group defined over $k$, then the group $G(k)$ of $k$-rational points is absolutely connected. In particular, all classical groups such as special linear groups $\SL_n(k)$ or symplectic groups $\Sp_{2n}(k)$ are absolutely connected.


We will use the fact that absolutely connected groups are perfect \cite[Theorem 3.5]{abscon}. In some arguments, we also consider absolutely connected groups of finite commutator width. 
Recall that all groups of the form $G(k)$ from the previous paragraph have finite commutator width. 

A few observations in this paper rely on a version of Beth's definability theorem for types rather than for formulas. For the reader's convenience, we give a precise formulation of this theorem. To prove it, one should modify the proof of Beth's theorem for formulas \cite[Theorem 2.2.22]{ki_cha}, which we leave as an exercise.


Let $L$ be a first order language. Recall that for a first order $L$-theory $T$ and collections of formulas $p_1(x)$ and $p_2(x)$ in the language $L$, the expression $T \models p_1(x) \equiv p_2(x)$ means that for any model $M \models T$ the types $p_1(x)$ and $p_2(x)$ have the same sets of realizations in $M$, i.e., $p_1(M) = p_2(M)$. Equivalently, for every $\varphi_1(x) \in p_1(x)$ there exists a conjunction $\varphi_2(x)$ of formulas from $p_2(x)$ such that $T \vdash \varphi_2(x) \rightarrow \varphi_1(x)$, and conversely, for every $\psi_2(x) \in p_2(x)$ there exists a conjunction $\psi_1(x)$ of formulas from $p_1(x)$ such that $T \vdash \psi_1(x) \rightarrow \psi_2(x)$.

\begin{fact}[Beth's theorem for types]\label{fact:Beth for types}
Assume $L$ is a first order language. Let $L' = L \cup \{ P_i',f_j',c_k': i \in I, j \in J, k \in K\}$ and $L''=L \cup \{ P_i'',f_j'',c_k'': i \in I, j \in J, k \in K\}$, where $P_i', P_i'',f_j',f_j'',c_k',c_k''$ are pairwise distinct relation, function and constant symbols, not appearing in the language $L$, and such that for each $i \in I$ the relation symbols $P_i'$ and $P_i''$ have the same arity and for each $j\in J$ the function symbols $f_j'$ and $f_j''$ have the same arity. Let $T'$ be a theory in $L'$ and $T''$ the corresponding theory in $L''$ (i.e., $T''$ is obtained from $T'$ by replacing each of the symbols $P_i', f_j',c_k'$ by $P_i'',f_j'',c_k''$, respectively). Let $\pi'(x)$ be a collection of formulas in $L'$ and $\pi''(x)$ the corresponding collection of formulas in $L''$. Assume $T' \cup T'' \models \pi'(x) \equiv \pi''(x)$. Then there exists a collection of formulas $\pi(x)$ in the language $L$ such that $T' \models \pi'(x) \equiv \pi(x)$.
\end{fact}

\section{Connected components and 2-cocycles} \label{sec:main}

This section is constructed as follows. After a short introduction on group extensions and 2-cocycles, we prove the main theorem of this paper (Theorem \ref{thm:main}), describing a general situation in which an extension $\widetilde{G}$ of a group $G$ by an abelian group $A$, defined in terms of a definable 2-cocycle with finite image, satisfies $\widetilde{G^*}^{000}_B \ne \widetilde{G^*}^{00}_B$ for some parameter set $B$ ($*$ means that we consider the sets of realizations in a monster model). Then, we obtain two essentially different corollaries of the main theorem, which are used in Section \ref{sec:central} to produce new classes of examples of groups of the form $\widetilde{G^*}$ satisfying $\widetilde{G^*}^{000}_B \ne \widetilde{G^*}^{00}_B.$
Next, we make a closer analysis of the situation from Theorem \ref{thm:main}, proving that in a rather general context the main assumption of this theorem is also a necessary condition for $\widetilde{G^*}^{000}_B \ne \widetilde{G^*}^{00}_B$ (Corollary \ref{cor:new 1}), and describing the quotient $\widetilde{G^*}^{00}_B/\widetilde{G^*}^{000}_B$. 


Let $G$ be an arbitrary group and let $A$ be an abelian group.
Assume that $\widetilde{G}$ is an extension of $G$ by $A$ (not necessarily central), i.e., we assume that there exists an exact sequence
\begin{equation}
\xymatrix{
 1 \ar@{^{(}->}[r] & A \ar@{^{(}->}[r] & {}\widetilde{G} \ar@{>>}[r]^-{\pi} & G \ar@{>>}[r] & 1
}. \label{eq:1}
\end{equation}
Sometimes by a group extension of $G$ by $A$ we mean the above sequence (and not just the group $\widetilde{G}$), which should be clear from the context. Then, $\widetilde{G} \cong (A\times G,*)$, where \[(a_1,g_1)*(a_2,g_2) = (a_1+g_1\cdot a_2+h(g_1,g_2),g_1g_2)\] for the action $\cdot\colon G \times A \to A$ of $G$ on $A$ by automorphisms induced by the conjugation action of $\widetilde{G}$ on $A$, and for some \emph{2-cocycle} $h$ (also called a \emph{factor set}), that is a map $h\colon G\times G \to A$ satisfying \cite[10.13]{rotman}:
\begin{itemize}
\item $h(g_1,g_2) + h(g_1g_2,g_3) = h(g_1,g_2g_3) + g_1\cdot h(g_2,g_3)$ for all $g_1,g_2,g_3\in G$,
\item $h(g,e)=h(e,g)=0$ for all $g\in G$.
\end{itemize}
%
%
%
More precisely, the extension \eqref{eq:1} is equivalent to the natural extension 
\begin{equation}
\xymatrix{
 1 \ar@{^{(}->}[r] & A \ar@{^{(}->}[r] & (A \times G, *) \ar@{>>}[r]^-{\pi} & G \ar@{>>}[r] & 1
}, \label{eq:2}
\end{equation}
that is, there exists an isomorphism $\varphi\colon \widetilde{G} \overset{\cong}{\to} (A\times G,*)$ such that the following diagram commutes
\[
\xymatrix{
A \ar@{^{(}->}[r] \ar[d]^-{\id} & {}\widetilde{G} \ar@{>>}[r]^-{\pi} \ar[d]^-{\cong}_{\varphi} & G \ar[d]^-{\id}\\
A \ar@{^{(}->}[r] & (A \times G, *) \ar@{>>}[r]^-{\pi'} & G.}
\]
%
%
Recall also that the 2-cocycles $h\colon G \times G \to A$ yielding extensions equivalent to \eqref{eq:1} are exactly the functions given by the formula $h(g_1,g_2)=s(g_1)s(g_2)s(g_1g_2)^{-1}$ with $s$ ranging over all sections of $\pi$ in \eqref{eq:1} (i.e., $s\colon G \to \widetilde{G}$ satisfying $\pi \circ s= \id_G$) such that $s(e_G)=e_{\widetilde{G}}$.

Conversely, for any action $\cdot$ of $G$ on $A$ by automorphisms and for any 2-cocycle $h\colon G \times G \to A$, the structure $(A\times G,*)$ is a group being an extension of $G$ by $A$. 
The identity element is $(0, e)$ and the inverse of $(a, g)$ equals $(-g^{-1}\cdot a - h(g^{-1},g), g^{-1})$.

From now on, $\widetilde{G}$ will denote $(A\times G,*)$ for some 2-cocycle $h$. We will freely identify $A$ with the subgroup $A \times \{ e \}$ of $\widetilde{G}$.

We say that the 2-cocycle $h\colon G \times G \to A$ is \emph{split} if there exists a function $f\colon G \to A$ for which:
\begin{itemize}
\item $h(g_1,g_2) = f(g_1) + g_1\cdot f(g_2)-f(g_1g_2)$ for all $g_1,g_2 \in G$,
\item $f(e)=0$.
\end{itemize}
In this situation, we also say that $h$ is \emph{split via the function $f$}. 

Recall that the 2-cocycle $h$ is split if and only if the extension \eqref{eq:1} is equivalent to the semidirect product extension of $G$ by $A$ \cite[10.15]{rotman}. In particular, if the action of $G$ on $A$ is trivial, then $h$ is split if and only if the extension \eqref{eq:1} is equivalent to the product extension. 

We say that 2-cocycles $h$ and $h'$ are \emph{cohomologous} via a function $f$ if $h-h'$ is split via $f$.


If $H \leq G$, then $h$ induces the restricted 2-cocycle $h_{| H \times H} \colon H \times H \to A$. If $C \leq A$ is invariant under the action of $G$, then the action of $G$ on $A$ induces an action of $G$ on $A/C$, and the 2-cocycle $h$ induces a 2-cocycle $\overline{h} \colon G \times G \to A/C$ in the obvious way. 

We consider a situation when the groups $G$ and $A$ and the action of $G$ on $A$ are $\emptyset$-definable in a (many-sorted) structure $\G$ (e.g. $\G$ consists of the pure groups $G$ and $(A,+)$ together with the action of $G$ on $A$).
We assume that the image $\im(h)$ of the 2-cocycle $h\colon G \times G \to A$ is finite and that $h$ is definable in $\G$ (the definability of $h$ is equivalent to the fact that the preimage by $h$ of any element of $A$ is a definable (in $\G$) subset of $G \times G$). 

The group $\widetilde{G}$ is, of course, definable in $\G$. Let $\G^* \succ \G$ be a monster model. Denote by $G^*$ the interpretation of $G$ in $\G^*$, by $A^*$ the interpretation of $A$ in $\G^*$, and by $\widetilde{G^*}$ the interpretation of $\widetilde{G}$ in $\G^*$. We have the following exact sequence
%
%
\begin{equation}
\xymatrix{
 1 \ar@{^{(}->}[r] & A^* \ar@{^{(}->}[r] & {}\widetilde{G^*} \ar@{>>}[r]^-{\pi} & G^* \ar@{>>}[r] & 1}, \label{eq:3}
\end{equation}
where $\pi$ is the projection on the second coordinate. The interpretations in $\G^*$ of various definable objects will be usually denoted by adding $^*$ (as above).
%

By ${G^*}^{000}_B$ and $\widetilde{G^*}^{000}_B$ we denote the smallest subgroups of bounded index of $G^*$ and $\widetilde{G^*}$, respectively, which are invariant under $\aut\left(\G^*/B\right)$ for a fixed parameter set $B$. By ${G^*}^{00}_B$ and $\widetilde{G^*}^{00}_B$ we denote the smallest subgroups of bounded index of $G^*$ and $\widetilde{G^*}$, respectively, which are type-definable in $\G^*$ over $B$, and similarly with ${G^*}^0_B$ and $\widetilde{G^*}^{0}_B$. 

By $A_0$ we will denote the subgroup of $A$ generated by the image of $h$. Notice that $A_0$ is countable and invariant under the action of $G^*$ (which easily follows from the first formula in the definition of 2-cocycles).

Let us distinguish the following two situations which will be the context of various results in this paper.



\begin{itemize}
\item[$(0)$] Let $G$ be a group acting by automorphisms on an abelian group $A$, where $G$, $A$ and the action of $G$ on $A$ are $\emptyset$-definable in a structure $\G$, and let $h\colon G \times G \to A$ be a 2-cocycle which is $B$-definable in $\G$ and with finite image $\im(h)$ contained in $\dcl(B)$ (the definable closure of $B$) for some finite parameter set $B\subset \G$. By $A_0$ we denote the subgroup of $A$ generated $\im(h)$. Then $A_0 \subseteq \dcl(B)$.
\item[$(00)$] Assume we are in Situation $(0)$. Additionally,  let $A^*_1$ be a bounded index subgroup of $A^*$ which is type-definable over $B$ and which is invariant under the action of $G^*$. 
\end{itemize} 

Let us  list two special cases in which the assumption that $A^*_1$ is invariant under $G^*$ is satisfied.

\begin{itemize}
\item Consider the case when $A^*_1 ={A^*}^0$ is the connected component of the pure group $(A^*,+)$, i.e., it is the intersection of all definable (in the pure group $(A^*,+)$) finite index subgroups of $A^*$. Then $A^*_1$ has bounded index in $A^*$, it is type-definable in $\G^*$ over $\emptyset$, and it is invariant under the action of $G^*$. 
\item If the action of $G$ on $A$ is trivial (i.e., the extension of $G$ by $A$ is central), then the assumption that $A^*_1$ is invariant under $G^*$ is clearly satisfied for every $A^*_1$.
\end{itemize}

We start from the following very useful observation.
\begin{remark}\label{rem:trivial action}
Consider Situation $(0)$.
\begin{enumerate} 
\item[(i)] ${G^*}^{0}_B$ acts trivially on $A_0$.
\item[(ii)] ${G^*}^{000}_B$ acts trivially on $A^*/ \left(\widetilde{G^*}^{000}_B \cap A^*\right)$.
\item[(iii)] Consider Situation $(00)$. Then ${G^*}^{00}_B$ acts trivially on $A^*/A^*_1$.
\end{enumerate}
\end{remark}

\begin{proof}
(i) 
Take any $a \in A_0$. Since $A_0$ is closed under the action of $G^*$ and countable, the index $\left[G^*:\Stab_{G^*}(a)\right]$ is also countable, where $\Stab_{G^*}(a)$ is the stabilizer of $a$ in $G^*$. As $a \in \dcl(B)$, we have that $\Stab_{G^*}(a)$ is $B$-definable with $\left[G^*:\Stab_{G^*}(a)\right]$ finite, and so $\Stab_{G^*}(a) \geq {G^*}^{0}_B$. Thus, ${G^*}^{0}_B$ acts trivially on $A_0$.\\
(ii) 
Consider any element $\overline{a}$ in $A^* / \left(\widetilde{G^*}^{000}_B \cap A^*\right)$. 
Since $\widetilde{G^*}^{000}_B$ is a normal subgroup of $\widetilde{G^*}$, the action of $G^*$ on $A^*$ induces a $B$-invariant action of $G^*$ on the quotient $A^* / \left(\widetilde{G^*}^{000}_B \cap A^*\right)$. As this quotient has bounded size, the orbit of $\overline{a}$ under the action of $G^*$ is also of bounded size. Thus, the stabilizer $\Stab_{G^*}(\overline{a})$ of $\overline{a}$ in $G^*$ is a bounded index subgroup of $G^*$ (invariant over $B,\overline{a}$). The intersection of all such stabilizers $\Stab_{G^*}(\overline{a})$ for $\overline{a}$ ranging over $A^* / \left(\widetilde{G^*}^{000}_B \cap A^*\right)$ is a $B$-invariant, bounded index subgroup of $G^*$, and so it contains ${G^*}^{000}_B$. Thus, ${G^*}^{000}_B$ acts trivially on $A^* / \left(\widetilde{G^*}^{000}_B \cap A^*\right)$.\\
(iii) A similar proof works.
\end{proof}

Whenever we say that a 2-cocycle is non-split via $B$-invariant functions, we mean that it is split via no $B$-invariant function.

Now, we turn to the main theorem of the paper. 

\begin{theorem} \label{thm:main}
Consider Situation $(00)$. Assume that:
\begin{enumerate} 
\item[(i)] the induced 2-cocycle $\overline{h^*}_{|{G^*}^{00}_B \times {G^*}^{00}_B}\colon {G^*}^{00}_B \times {G^*}^{00}_B \to A_0/\left({A^*}_{1} \cap A_0\right)$ is non-split via $B$-invariant functions,
\item[(ii)] $A_0/\left( {A^*}_{1} \cap A_0 \right)$ is torsion free (and so isomorphic with $\Z^n$ for some natural $n$).
\end{enumerate}
Then $\widetilde{G^*}^{000}_B \ne \widetilde{G^*}^{00}_B$.
%

Suppose furthermore that ${G^*}^{000}_B = G^*$, and for every proper, type-definable over $B$ in $\G^*$ and invariant under the action of $G^*$ subgroup $H$ of $A^*$ of bounded index the induced 2-cocycle $\overline{h^*}\colon G^* \times G^* \to A_0/\left(H \cap A_0\right)$ is non-split via $B$-invariant functions. Then $\widetilde{G^*}^{00}_B = \widetilde{G^*}$.
\end{theorem}
%
%
%
Before the proof, we would like to note that Remark \ref{rem:trivial action}(i) implies that the action of ${G^*}^{00}_B$ on $A_0$ is trivial, so the notion of  induced 2-cocycle in Assumption (i) of the above theorem makes sense, and similarly in the last part of the theorem. 
Notice also that if the assumptions of the theorem are satisfied in the monster model $\G^*$, then they are satisfied in any bigger monster model.

\begin{proof}
First, we prove the following claim.

\begin{clm} Suppose $H \leq A^*$ is a subgroup of bounded index which is invariant both under $\aut(\G^*/B)$ and under the action of $G^*$ on $A^*$. Then $(H + A_0) \times {G^*}^{000}_B$ is a subgroup of $\widetilde{G^*}$ containing $\widetilde{G^*}^{000}_B$.
\end{clm}
\begin{proof}[Proof of Claim 1]
The fact that $(H + A_0) \times {G^*}^{000}_B$ is a subgroup of $\widetilde{G^*}$ follows from the invariance of $H$ under the action of $G^*$ and the observations that the image $\im(h^*)$ is contained in $A_0$ and $A_0$ is closed under the action of $G^*$ (which follows from the first formula in the definition of 2-cocycles). Moreover, it is clear that $(H + A_0) \times {G^*}^{000}_B$ is $B$-invariant and has bounded index in $\widetilde{G^*}$. This shows the desired inclusion.
\end{proof}

It is easy to see that ${G^*}^{000}_B=\pi \left[\widetilde{G^*}^{000}_B\right]$ (in \eqref{eq:3}). Let $H \leq A^*$ be as in Claim 1. As a consequence of Claim 1, 
we have that the following sequence is exact:
\begin{equation}
\xymatrix{
 1 \ar@{^{(}->}[r] & \left(H+A_0\right)\cap\widetilde{G^*}^{000}_B \ar@{^{(}->}[r] & {}\widetilde{G^*}^{000}_B \ar@{>>}[r]^-{\pi} & {G^*}^{000}_B \ar@{>>}[r] & 1.} \label{eq:4}
\end{equation}
%
%
We can say even more about this situation. Notice that if $H$ satisfies the assumptions of Claim 1, then so does $H \cap \widetilde{G^*}^{000}_B$, because $\widetilde{G^*}^{000}_B$ is a normal subgroup of $\widetilde{G^*}$. Thus, Claim 1 yields the following exact sequence
\begin{equation} 
\xymatrix{1 \ar@{^{(}->}[r] & \left(H\cap \widetilde{G^*}^{000}_B +A_0\right)\cap\widetilde{G^*}^{000}_B \ar@{^{(}->}[r] & {}\widetilde{G^*}^{000}_B \ar@{>>}[r]^-{\pi} & {G^*}^{000}_B \ar@{>>}[r] & 1,}
\end{equation} 
and so $\ker\left(\pi_{|\widetilde{G^*}^{000}_B}\right)=H \cap \widetilde{G^*}^{000}_B + A_0 \cap \widetilde{G^*}^{000}_B$.
\begin{clm} Let $H \leq A^*$ be as in Claim 1. If $A_0\cap\widetilde{G^*}^{000}_B\subseteq H$, then the induced 2-cocycle \[\overline{h^*}_{|{G^*}^{000}_B \times {G^*}^{000}_B}\colon {G^*}^{000}_B \times {G^*}^{000}_B \to A_0/\left(H \cap A_0\right)\] is split via a $B$-invariant function.
\end{clm}
\begin{proof}[Proof of Claim 2]
If $A_0\cap\widetilde{G^*}^{000}_B\subseteq H$, then the above conclusion gives us \[\ker\left(\pi_{|\widetilde{G^*}^{000}_B}\right)= H\cap\widetilde{G^*}^{000}_B.\] 
By Claim 1, take a section $s\colon{G^*}^{000}_B \to \widetilde{G^*}^{000}_B$ of $\pi$ of the form $s(g)=(a_g,g)$, where each $a_g\in H+A_0$ and $a_e=0$. Then, \[a_g = b_g + c_g,\] for some $b_g\in H$ and $c_g\in A_0$. Since $A_0 \subseteq \dcl(B)$, it is clear that section $s$ can be chosen so that the function $f \colon {G^*}^{000}_B \to A_0$ defined by $f(g)=-c_g$ is $B$-invariant. 

Consider the 2-cocycle $h'\colon {G^*}^{000}_B \times {G^*}^{000}_B \to A^*$ defined by \[h'(g_1,g_2)=s(g_1)s(g_2)s(g_1g_2)^{-1}.\] It takes values in $\ker\left(\pi_{|\widetilde{G^*}^{000}_B}\right)=H\cap\widetilde{G^*}^{000}_B$. Moreover, it is cohomologous to $h^*_{|{G^*}^{000}_B \times {G^*}^{000}_B}$, because, using the fact that $g\cdot h^*(g^{-1},g)=h^*(g,g^{-1})$ (which follows from the definition of 2-cocycles), we have
\begin{eqnarray*}
 h'(g_1,g_2) &=& (a_{g_1},g_1)(a_{g_2},g_2)(a_{g_1g_2},g_1g_2)^{-1} \\
   &=& h^*(g_1,g_2)+g_1\cdot a_{g_2}-a_{g_1g_2}+a_{g_1}\\
 &+&h^*\left(g_1g_2,(g_1g_2)^{-1}\right)-g_1g_2\cdot h^*\left((g_1g_2)^{-1},g_1g_2\right) \\
   &=& h^*(g_1,g_2)+g_1\cdot a_{g_2}-a_{g_1g_2}+a_{g_1}.
\end{eqnarray*}
The above equality implies that \[h'(g_1,g_2) - g_1\cdot b_{g_2}+b_{g_1g_2}-b_{g_1} = h^*(g_1,g_2)+g_1\cdot c_{g_2}-c_{g_1g_2}+c_{g_1}.\] Both sides of this equality define the same 2-cocycle $h''(g_1,g_2)$, which takes values in $H\cap A_0$ (because the left hand side takes values in $H$, and the right hand side takes values in $A_0$). Thus, $h^*_{|{G^*}^{000}_B \times {G^*}^{000}_B}$ is cohomologous to a 
2-cocycle with values in $H\cap A_0$ via the $B$-invariant function $f\colon {G^*}^{000}_B \to A_0$ (defined by $f(g)=-c_g$).
\end{proof}

To prove the first part of the theorem, suppose for a contradiction that $\widetilde{G^*}^{000}_B = \widetilde{G^*}^{00}_B$. Put $H=\widetilde{G^*}^{000}_B \cap A^*_1$. Then, $H$ is a subgroup of bounded index of $A^*$, which is type-definable in $\G^*$ over $B$ and invariant under the action of $G^*$ (the last fact follows from the observation that $\widetilde{G^*}^{000}_B$ is a normal subgroup of $\widetilde{G^*}$). By Claim 1, we get that $(H + A_0) \times {G^*}^{000}_B$ is a subgroup of $\widetilde{G^*}$ containing $\widetilde{G^*}^{000}_B$. Thus, $\widetilde{G^*}^{000}_B \cap A^* \leq H + A_0$.
On the other hand, by Claim 2 together with the assumption (i) and the fact that ${G^*}^{00}_B=\pi \left[ \widetilde{G^*}^{00}_B\right]=\pi \left[ \widetilde{G^*}^{000}_B\right]={G^*}^{000}_B$, we get $A_0\cap\widetilde{G^*}^{000}_B\not\subseteq H$.
%
%
Using the assumption that $A_0/\left( A^*_1 \cap A_0 \right)$ is torsion free and the fact that $A_0$ is countable, we conclude that $\left(\widetilde{G^*}^{000}_B \cap A^*\right)/H$ is countably infinite. But, since $\widetilde{G^*}^{000}_B = \widetilde{G^*}^{00}_B$, the group $\widetilde{G^*}^{000}_B \cap A^*$ is type-definable in $\G^*$, and so $\left(\widetilde{G^*}^{000}_B \cap A^*\right)/H$ is a compact topological group, and as such, it cannot be of cardinality $\aleph_0$, a contradiction.

Now, we prove the second part of the theorem. Let $H = \widetilde{G^*}^{00}_B\cap A^*$. By Claim 2 and our assumption, $H = A^*$. Therefore, $A^*\leq \widetilde{G^*}^{00}_B$. On the other hand $\pi\left[\widetilde{G^*}^{00}_B \right] \supseteq \pi\left[\widetilde{G^*}^{000}_B \right] ={G^*}^{000}_B=G^*$. Hence, $\widetilde{G^*}^{00}_B= \widetilde{G^*}$.
\end{proof}

A natural goal is to simplify Assumption (i) of Theorem \ref{thm:main} in order to make it easier to verify in concrete applications. We do this in the next corollary.

\begin{corollary} \label{cor:main}
Consider Situation $(00)$. Assume that:
\begin{enumerate}
\item the 2-cocycle $h\colon G\times G \to A_0$ is non-split (via a function taking values in $A_0$),
\item ${A^*}_{1} \cap A_0$ is trivial and $A_0$ is torsion free (and so $A_0\cong \Z^n$ for some $n$),
\item ${G^*}^{00}_B = G^*$.
\end{enumerate}
Then $\widetilde{G^*}^{000}_B \ne \widetilde{G^*}^{00}_B$.

Suppose furthermore that ${G^*}^{000}_B = G^*$ and:
\begin{enumerate}
\item[(4)] (strong non-splitness of $h$) for every proper subgroup $Z \lneqq A_0=\Z^n$ the induced 2-cocycle $\overline{h} \colon G \times G \to \Z^n/Z$ is non-split,
\item[(5)] (denseness of $A_0$ in $A^*$) if $H \leq A^*$ is a type-definable over $B$ (in $\G^*$) subgroup of $A^*$ of bounded index and containing $A_0$, then $H=A^*$.
\end{enumerate}
Then $\widetilde{G^*}^{00}_B = \widetilde{G^*}$.
\end{corollary}


Before the proof, we give some comments on the assumptions of this corollary. Notice that since $A_0$ is finitely generated, $(2)$ implies that $A_0\cong \Z^n$ for some natural number $n$. By Remark \ref{rem:trivial action}(i), every subgroup $Z$ considered in (4) is invariant under the action of $G^*$, and so it makes sense to talk about the induced 2-cocycle $\overline{h}\colon G \times G \to \Z^n/Z$. 
Another remark is that every definably absolutely connected group (see Definition \ref{def:abscon}) satisfies ${G^*}^{000}_B = G^*$, so also $(3)$.
Finally, we explain why (5) was called `denseness of $A_0$ in $A^*$'. Let ${A^*}^{00}_B$ be the smallest subgroup of bounded index of $A^*$ which is type-definable over $B$ in $\G^*$. It is easy to check that $(5)$ is equivalent to the fact that $A_0/{A^*}^{00}_B$ is a dense subset of the topological group $A^*/{A^*}^{00}_B$.

\begin{proof}
We have to prove that the assumptions of Theorem \ref{thm:main} are satisfied. 
First note that by (2), the group $A_0/\left( {A^*}_{1} \cap A_0\right) = A_0$ is torsion free.
Suppose for a contradiction that Assumption (i) of Theorem \ref{thm:main} does not hold. This implies, by (2) and (3), that the 2-cocycle $h^*\colon G^* \times G^* \to A_0$ is split.
Then, after restriction to $G$, we get that $h\colon G \times G \to A_0$ is split (via a function taking values in $A_0$), a contradiction to (1).

The second part of the corollary holds, because $(4)$ and $(5)$ imply that for every proper, type-definable over $B$ (in $\G^*$) subgroup $H$ of $A^*$ of bounded index the induced 2-cocycle $\overline{h}\colon G \times G \to A_0/\left(H \cap A_0\right)$ is non-split which implies that  $\overline{h^*}\colon G^* \times G^* \to A_0/\left(H \cap A_0\right)$ is non-split.
\end{proof}

Next, we notice that the examples from \cite[Section 3]{conv_pillay} follow from the above corollary.




\begin{example}\label{ex:PC1}
Let $\G = ((\Z,+), (\R,+,\cdot,<,0,1))$, $G=\SL_2(\R)$ and $A=(\Z,+)$ ($G$ and $A$ are $\emptyset$-definable in $\G$, $G$ acts trivially on $A$). Let $\widetilde{G}=\widetilde{\SL_2(\R)}$ be the topological universal cover of $\SL_2(\R)$. $\widetilde{\SL_2(\R)}$ is defined by means of the 2-cocycle $h\colon G\times G \to \Z$ considered in \cite[Theorem 2]{asai}. We recall the definition of $h$ from \cite{asai}, from which it is clear that $h$ is $B$-definable in $\G$ for $B:=\{1\}$ 
(where 1 is from the sort $(\Z,+)$), and which will be crucial in the last part of Section \ref{sec:central}. 
For $c,d\in\R$ define the following symbol $c(d) = 
\left\{
  \begin{array}{l l}
    c &: c\ne 0 \\
    d &: c=0 \\ 
  \end{array}
\right..$
Consider any 
$\left( 
\begin{array}{cc}
a_1 & b_1 \\
c_1 & d_1 \\ 
\end{array}\right), \left( 
\begin{array}{cc}
a_2 & b_2 \\
c_2 & d_2 \\ 
\end{array}\right) \in \SL_2(\R)$. Let $\left( 
\begin{array}{cc}
a_3 & b_3 \\
c_3 & d_3 \\ 
\end{array}\right)=\left( 
\begin{array}{cc}
a_1 & b_1 \\
c_1 & d_1 \\ 
\end{array}\right) \left( 
\begin{array}{cc}
a_2 & b_2 \\
c_2 & d_2 \\ 
\end{array}\right)$. Then

\[h\left(\left( 
\begin{array}{cc}
a_1 & b_1 \\
c_1 & d_1 \\ 
\end{array}\right), \left( 
\begin{array}{cc}
a_2 & b_2 \\
c_2 & d_2 \\ 
\end{array}\right) \right) = \left\{
  \begin{array}{l l}
    1 &: c_1(d_1)>0 \wedge c_2(d_2)>0 \wedge c_3(d_3)<0, \\
    -1 &: c_1(d_1)<0 \wedge c_2(d_2)<0 \wedge c_3(d_3)>0, \\ 
    0 &: \text{otherwise.} \\
  \end{array}
\right..\]
Put as $A^*_1$ the connected component ${\Z^*}^0$ of the pure group $(\Z^*,+)$. 
Then the assumptions of the above corollary are satisfied, so we get $\widetilde{G^*} = \widetilde{G^*}^{00} \ne \widetilde{G^*}^{000}$, where $\widetilde{G^*}$ is the interpretation of $\widetilde{G}$ in a monster model $\G^* = ((\Z^*,+),(\R^*,+,\cdot,<,0,1))$.


\end{example}
\begin{proof} The assumptions (1) and (4) of Corollary \ref{cor:main} are true by \cite[Theorem 1]{asai}, (3) follows from the absolute connectedness of $\SL_2(\R)$, and $A_0=\Z$ implies (2).
The condition (5) follows from the observation that if $D\subseteq \Z^*$ is a $B$-definable subset containing $\Z$ and $B\subseteq \Z$, then $D=\Z^*$. Notice that since $\G$ has NIP, we can skip the parameter set $B$ in all connected components.
%
\end{proof}


\begin{example}\label{ex:PC2} Let $\G = (\R,+,\cdot,<,0,1)$, $G=\SL_2(\R)$ and $A=\SO_2(\R)$. Assume that the action of $G$ on $A$ is trivial. Fix a non-torsion element $g\in A$. Let $\widetilde{G}$ be defined by means of the 2-cocycle $h'\colon G \times G \to \SO_2(\R)$ defined as $h'(x,y)=h(x,y)g$ for $h$ from Example \ref{ex:PC1}. The group $\widetilde{G}$ is definable in $\G$. Put $A^*_1= \SO_2(\R^*)^{00}$ and $B=\{g\}$. Then, as in the previous example, the assumptions of Corollary \ref{cor:main} are satisfied, so we get $\widetilde{G^*} = \widetilde{G^*}^{00} \ne \widetilde{G^*}^{000}$.
\end{example}

Using Corollary \ref{cor:main} and Matsumoto-Moore theory, we find in Subsection \ref{subsection:4.1} much more general classes of examples.

Now, we will deduce another corollary of Theorem \ref{thm:main}, whose proof is a bit surprising, as we start from the 2-cocycle $h$ which is split on the original group $G$ in order to get the non-splitness of the 2-cocycle considered in Assumption (i) of the theorem. This corollary leads in Subsection \ref{subsection:4.2} to various examples which were not accessible by the methods used in \cite{conv_pillay} or by Corollary \ref{cor:main}. In order to get these examples, we use various quasi-characters considered in bounded cohomology.

\begin{corollary}\label{cor:main2}
Consider Situation $(00)$. Assume that: 
\begin{enumerate}
\item[(1)] there is a $B$-definable subgroup $H$ of $G$ of finite index such that $h_{|H \times H} \colon H \times H \to A$ is split via a $B$-definable function $f$,
\item[(2)] Assumption $(ii)$ of Theorem \ref{thm:main} holds,
\item[(3)] $\overline{f^*}\left[{G^*}^{00}_B\right] \nsubseteq A_0/A_1^*$, where  $\overline{f^*}\colon G^* \to A^*/A_1^*$ is induced by the interpretation $f^*$ of $f$ in $\G^*$. 
\end{enumerate}
Then $\widetilde{G^*}^{000}_B \ne \widetilde{G^*}^{00}_B$.

Moreover, under (1) and (2), Assumption (3) is satisfied when one (or both) of the following conditions holds:
\begin{enumerate}
\item[(4)] the image $\overline{f^*}\left[{G^*}^{00}_B\right]$ contains a non-trivial subgroup of $A^*/A_1^*$,
\item[(5)] there is $a \in A \setminus A^*_1$ such that for any [$B$-definable] thick subset $P$ of $G$ there exists $g \in P\cap H$ such that for every $n \in \Z$, $f(g^n)=na$.
\end{enumerate}
\end{corollary}

\begin{proof}
If ${G^*}^{000}_B \ne {G^*}^{00}_B$, then $\widetilde{G^*}^{000}_B \ne \widetilde{G^*}^{00}_B$, so  we can assume that ${G^*}^{000}_B = {G^*}^{00}_B$. To prove the first part of the corollary, it is enough to show that Assumption (i) of Theorem \ref{thm:main} is satisfied. Suppose for a contradiction that  $\overline{h^*}_{|{G^*}^{00}_B \times {G^*}^{00}_B} \colon {G^*}^{00}_B \times {G^*}^{00}_B \to A_0/A^*_1$ is split via a $B$-invariant function $g \colon {G^*}^{00}_B \to A_0/A^*_1$ (here, we identify $A_0/A_0 \cap A^*_1$ with $A_0/A^*_1$ by the obvious $B$-invariant isomorphism). By Remark \ref{rem:trivial action}(i) or (iii), this means that $\overline{h^*}(x,y)=g(x) +g(y)-g(xy)$ for all $x,y \in {G^*}^{00}_B$. Since by (1) and Remark \ref{rem:trivial action}(iii) also $\overline{h^*}(x,y)=\overline{f^*}(x) + \overline{f^*}(y) - \overline{f^*}(xy)$ for all $x,y \in {G^*}^{00}_B$, we get that \[\overline{f^*}-g \colon {G^*}^{00}_B  \longrightarrow A^*/A^*_1\] is a $B$-invariant homomorphism. By (3) and the fact that $\im(g) \subseteq A_0/A^*_1$, we get that $\overline{f^*}_{|{G^*}^{00}_B} \ne g$. Thus, $\ker\left(\overline{f^*}-g\right)$ is a proper $B$-invariant subgroup of ${G^*}^{00}_B={G^*}^{000}_B$, but the index $\left[{G^*}^{000}_B:\ker\left(\overline{f^*}-g\right)\right]\leq |A^*/A^*_1|$ is bounded, which contradicts the definition of ${G^*}^{000}_B$.

To see the `moreover' part of the corollary, first assume (4). Denote by $L$ a non-trivial subgroup of $A^*/A^*_1$ contained in $\overline{f^*}\left[{G^*}^{00}_B\right]$. We can assume that $L \leq A_0/A^*_1$ (otherwise (3) holds). Then $L$ is infinite, as $A_0/A^*_1$ is torsion-free. Since $\overline{f^*}\left[{G^*}^{00}_B\right]$ is closed in the logic topology on $A^*/A^*_1$, the closure $\overline{L}$ is contained in $\overline{f^*}\left[{G^*}^{00}_B\right]$. But being an infinite, compact topological group,  $\overline{L}$ is uncountable, and so $\overline{f^*}\left[{G^*}^{00}_B\right] \nsubseteq A_0/A_1^*$ (as $A_0/A^*_1$ is countable), i.e., $(3)$ holds. Now, assume (5). By Fact \ref{fact:new 1}(ii) and compactness, we easily get that (4) is satisfied, so we are done.
\end{proof}


Although we are not going to use the language of cohomology groups in the current paper, it is worth to mention that Theorem \ref{thm:main} and its corollaries are related to bounded cohomology (see \cite{Gr}). For example, if $A=A_0=\Z$, then the 2-cocycle $h$ has finite image iff it is bounded. 
So, the class of such a 2-cocycle $h$ is an element of the second bounded cohomology group of $G$:
in Corollary \ref{cor:main}, we consider the situation when the 2-cocycle $h$ yields a non-trivial element in the non-singular part of the second bounded cohomology group, whereas in Corollary \ref{cor:main2}, it yields a trivial element in the non-singular part. It could be interesting to investigate relationships between properties of classes of 2-cocycles as elements of  the second bounded cohomology group and the properties of connected components in the corresponding group extensions.


Now, we will undertake a closer analysis of the situation from Theorem \ref{thm:main}. First of all, we show that, in a rather general context, Assumption (i) of Theorem \ref{thm:main} is not only sufficient but also necessary in order to have $\widetilde{G^*}^{000}_B \ne \widetilde{G^*}^{00}_B$. Then, we give a description of the quotient $\widetilde{G^*}^{00}_B/\widetilde{G^*}^{000}_B$, assuming that ${G^*}^{000}_B = {G^*}^{00}_B$. We also formulate some questions related to these issues.

Before we go to the details, notice that if one wants to find a necessary condition on the 2-cocycle $h$ from Theorem \ref{thm:main} for $\widetilde{G^*}$ to satisfy $\widetilde{G^*}^{000}_B \ne \widetilde{G^*}^{00}_B$, one should assume that ${G^*}^{000}_B = {G^*}^{00}_B$, as otherwise $\widetilde{G^*}^{000}_B$ is automatically different from $\widetilde{G^*}^{00}_B$.

\begin{proposition}\label{prop:x}
Consider Situation $(00)$.
\begin{enumerate}
\item Suppose  ${G^*}^{000}_B ={G^*}^{00}_B$. Assume additionally that $A^*_1 \subseteq \widetilde{G^*}^{000}_B \cap A^*$. Then, the condition $\widetilde{G^*}^{000}_B \ne \widetilde{G^*}^{00}_B$ implies $\widetilde{G^*}^{00}_B \cap A^* \not\subseteq A^*_1$. Under Assumption (ii) of Theorem \ref{thm:main}, the converse is also true.
\item The induced 2-cocycle 
$\overline{h^*}_{|{G^*}^{000}_B \times {G^*}^{000}_B} \colon {G^*}^{000}_B \times {G^*}^{000}_B \to A_0/(A^*_1 \cap A_0)$ is non-split via $B$-invariant functions if and only if $\widetilde{G^*}^{000}_B \cap A^* \not\subseteq A^*_1$.
\end{enumerate}
\end{proposition}

Before the proof, notice that the assumption $A^*_1 \subseteq \widetilde{G^*}^{000}_B \cap A^*$ cannot be dropped. 
Indeed, consider $A^*_1:=A^*$ to see that without this assumption, $\widetilde{G^*}^{00}_B \cap A^* \not\subseteq A^*_1$ does not need to hold. One can also take any example in which $\widetilde{G^*}^{000}_B \ne \widetilde{G^*}^{00}_B$ and define a new $A^*_1$ as $\widetilde{G^*}^{00}_B \cap A^*$. 

Notice, however, that the assumption $A^*_1 \subseteq \widetilde{G^*}^{000}_B \cap A^*$ holds in various interesting cases (e.g. if $ A^*_1={A^*}^{00}_B={A^*}^{000}_B$ and the action of $G$ on $A$ is trivial; in particular, if $\G=((A,+), M)$, $G$ is definable in the structure $M$ and it acts trivially on $A$, and $A^*_1={A^*}^{0}$).

\begin{proof}
(1) Under Assumption (ii) of Theorem \ref{thm:main}, $(\Leftarrow)$ follows by an argument similar to the one contained in the penultimate paragraph of the proof of Theorem \ref{thm:main} (and does not require the extra assumptions made in Proposition \ref{prop:x}(i)).
Namely, suppose for a contradiction that $\widetilde{G^*}^{000}_B = \widetilde{G^*}^{00}_B$. By Claim 1 in the proof of Theorem \ref{thm:main}, we know that $\widetilde{G^*}^{000}_B \cap A^* \subseteq A^*_1+A_0$, so $\left( \widetilde{G^*}^{000}_B\cap A^*\right)/A^*_1$ is countable, and, if it is non-trivial, it must be infinite (as $\left( A_0+A^*_1\right)/A^*_1$ is torsion free). But $\left( \widetilde{G^*}^{000}_B\cap A^*\right)/A^*_1$ is non-trivial, because $\widetilde{G^*}^{000}_B \cap A^* = \widetilde{G^*}^{00}_B \cap A^* \nsubseteq A^*_1$. Hence, $\left( \widetilde{G^*}^{00}_B \cap A^* \right) / A^*_1$ is a countable, infinite, compact group, which is impossible.

$(\Rightarrow)$. 
Suppose for a contradiction that $\widetilde{G^*}^{00}_B \cap A^* \subseteq A^*_1$.
Then $A^*_1=\widetilde{G^*}^{00}_B \cap A^* = \widetilde{G^*}^{000}_B \cap A^*$. 
Since ${G^*}^{000}_B={G^*}^{00}_B$, by the exactness of the sequences
\begin{equation}
\xymatrix{
 1 \ar@{^{(}->}[r] & {}\widetilde{G^*}^{00}_B \cap A^* \ar@{^{(}->}[r] & {}\widetilde{G^*}^{00}_B \ar@{>>}[r]^-{\pi} & {G^*}^{00}_B \ar@{>>}[r] & 1}, \label{eq:5}
\end{equation}
\begin{equation}
\xymatrix{
 1 \ar@{^{(}->}[r] & {}\widetilde{G^*}^{000}_B \cap A^* \ar@{^{(}->}[r] & {}\widetilde{G^*}^{000}_B \ar@{>>}[r]^-{\pi} & {G^*}^{000}_B \ar@{>>}[r] & 1}, \label{eq:6}
\end{equation}
we get $\widetilde{G^*}^{000}_B = \widetilde{G^*}^{00}_B$, which is a contradiction.\\[3mm]
(2) The implication $(\Rightarrow)$ follows from Claim 2 in the proof of Theorem \ref{thm:main}.

$(\Leftarrow)$. Since $A_0/(A^*_1 \cap A_0)$ can be identified with $A_0/A^*_1$ by a $B$-invariant isomorphism, $\overline{h^*}$ can be treated as a 2-cocycle with values in $A_0/A^*_1$. Suppose for a contradiction that $\overline{h^*}_{|{G^*}^{000}_B \times {G^*}^{000}_B}$ is split via a $B$-invariant function. By Remark \ref{rem:trivial action}(i), this means that $\overline{h^*}(x,y)=f(x)+f(y)-f(xy)$ for some $B$-invariant function $f \colon {G^*}^{000}_B \to A_0/A^*_1$ satisfying $f(e)=0+A^*_1$. 


Let $H$ be the extension of ${G^*}^{000}_B$ by $A^*/A^*_1$ defined by means of the 2-cocycle $\overline{h^*}_{|{G^*}^{000}_B \times {G^*}^{000}_B}$, with the action of ${G^*}^{000}_B$ on $A^*/A^*_1$ induced by the action of $G^*$ on $A^*$, which is trivial by Remark \ref{rem:trivial action}(iii). Let $K$ be the corresponding product $A^*/A^*_1 \times {G^*}^{000}_B$. Both groups $H$ and $K$ live in $\G^*$ as $B$-invariant objects.

A classical fact tells us that since $\overline{h^*}_{|{G^*}^{000}_B \times {G^*}^{000}_B}$ is split via $f$, the function $\Phi \colon H \to K$ given by $\Phi(a,g)=(a+f(g),g)$ is an isomorphism of groups. Since $f$ is $B$-invariant, so is $\Phi$.

Although $K$ is not definable but only $B$-invariant, we can still define $K^{000}_B$ as the smallest $B$-invariant subgroup of $K$ of bounded index. The trivial subgroup $A^*_1/A^*_1$ of $A^*/A^*_1$ will be denoted by $\{ 0+A^*_1 \}$. Similar remarks apply to $H$.

As $\{ 0 +A^*_1 \} \times {G^*}^{000}_B$ is a $B$-invariant subgroup of $K$ of bounded index, $K^{000}_B = \{ 0 +A^*_1\} \times {G^*}^{000}_B$. Hence, $K^{000}_B \cap \left( A^*/A^*_1 \right)= \{ 0 +A^*_1\}$. 

Using this together with the fact that $\Phi$ is a $B$-invariant isomorphism, we conclude that
%
\[
\begin{array}{lll}
\left\{ 0 +A^*_1\right\} &= \Phi^{-1}\left[\left\{ 0+A^*_1 \right\}\right] = \Phi^{-1}\left[ K^{000}_B \cap \left( A^*/A^*_1\right) \right]
= H^{000}_B \cap \left( A^*/A^*_1\right) \\
& = \left(\widetilde{G^*}^{000}_B \cap A^*\right)/ A^*_1.
\end{array}
\] 
Therefore, $\widetilde{G^*}^{000}_B \cap A^*\subseteq A^*_1$, a contradiction.
\end{proof}

Now, we will deduce from the last proposition that, in some general context, the conclusion of Theorem \ref{thm:main} implies its Assumption (i) (see Corollary \ref{cor:new 1}). In order to do that, we will need some results from \cite{Kr_Pi_So} and \cite{Ka_Mi_Si} on Borel cardinalities of Lascar strong types. Since Corollary \ref{cor:new 1} will follow formally from these results, we will not recall the relevant notions from \cite{Kr_Pi_So} in the current paper. Note that \cite{Kr_Pi_So} refers to the first version of the current paper (available on arXiv) which was written before the results from \cite{Ka_Mi_Si} were known.

Consider Situation $(00)$. Suppose ${G^*}^{000}_B ={G^*}^{00}_B$ and $A^*_1 \subseteq \widetilde{G^*}^{000}_B \cap A^*$. By Proposition \ref{prop:x}, in order to show that the conclusion of Theorem \ref{thm:main} (i.e., $\widetilde{G^*}^{000}_B \ne \widetilde{G^*}^{00}_B$) implies its Assumption (i) (i.e., $\overline{h^*}_{|{G^*}^{00}_B \times {G^*}^{00}_B} \colon {G^*}^{00}_B \times {G^*}^{00}_B \to A_0/(A^*_1\cap A_0)$ is non-split via $B$-invariant functions), it is enough to prove that
\begin{equation}
\widetilde{G^*}^{000}_B \cap A^*\subseteq A^*_1 \Longrightarrow \widetilde{G^*}^{00}_B \cap A^*\subseteq A^*_1.\label{eq:8}
\end{equation}

Suppose that this fails. Then $\widetilde{G^*}^{000}_B \cap A^* =A^*_1$ is type-definable but different from $\widetilde{G^*}^{00}_B \cap A^*$. Assuming that the language is countable,  \cite[Proposition 4.7]{Kr_Pi_So} tells us that in such a situation the relation on $\widetilde{G^*}^{00}_B$ of lying in the same coset modulo $\widetilde{G^*}^{000}_B$ (interpreted in the space of types over a small model) is non-trivial and smooth in the sense of Borel reducability. By \cite[Proposition 2.11]{Kr_Pi_So}, this implies that \cite[Conjecture 1]{Kr_Pi_So} is false, i.e., there exists an example of a Kim-Pillay type on which the relation of lying in the same Lascar strong type is non-trivial and smooth. However, very recently, Kaplan, Miller and Simon  proved \cite[Conjecture 1]{Kr_Pi_So} in \cite{Ka_Mi_Si}. This is a contradiction. Thus, we obtain the following corollaries.


\begin{corollary}\label{cor:new 1}
Consider Situation $(00)$ and assume that the language is countable. Suppose ${G^*}^{000}_B ={G^*}^{00}_B$ and $A^*_1 \subseteq \widetilde{G^*}^{000}_B \cap A^*$. Then, the conclusion of Theorem \ref{thm:main} implies its Assumption (i).
\end{corollary}

\begin{corollary}\label{cor:new 2}
Consider Situation $(0)$ and assume that the language is countable. Suppose ${G^*}^{000}_B ={G^*}^{00}_B$. Then, if $\widetilde{G^*}^{000}_B \cap A^*$ is type-definable, it must coincide with $\widetilde{G^*}^{00}_B \cap A^*$ (and so  $\widetilde{G^*}^{000}_B = \widetilde{G^*}^{00}_B$). 
\end{corollary}

We do not know whether (\ref{eq:8}) is true without the assumption  $A^*_1 \subseteq \widetilde{G^*}^{000}_B \cap A^*$ or the assumption that the language is countable.

\begin{conjecture}\label{con:new 1}
Consider Situation $(00)$. Suppose ${G^*}^{000}_B ={G^*}^{00}_B$. Then $\widetilde{G^*}^{00}_B \cap A^*\subseteq A^*_1 \iff \widetilde{G^*}^{000}_B \cap A^*\subseteq A^*_1$.
\end{conjecture}

This conjecture is equivalent to the following conjecture, generalizing Corollary \ref{cor:new 2}.

\begin{conjecture}\label{con:new 2}
Consider Situation $(0)$. Suppose ${G^*}^{000}_B ={G^*}^{00}_B$. Let $H$ be a $B$-invariant subgroup of $\widetilde{G^*}$ of bounded index such that $H \cap A^*$ is type-definable. Then $\widetilde{G^*}^{00}_B \cap A^*\subseteq H \cap A^*$.
\end{conjecture}

To show that Conjecture \ref{con:new 1} implies \ref{con:new 2}, first, by a standard trick, replace $H$ by a normal subgroup of $\widetilde{G^*}$, and then apply \ref{con:new 1} to $A^*_1:=H \cap A^*$ . To see the converse, apply \ref{con:new 2} to $H:=A^*_1+ \widetilde{G^*}^{000}_B$.

Conjecture \ref{con:new 1} is important, because it implies that Corollary \ref{cor:new 1} is true without the assumption that the language is countable and because it would allow us to understand better the quotient $\widetilde{G^*}^{00}_B/\widetilde{G^*}^{000}_B$ in the final part of this section (see Remark \ref{rem:zastosowanie 2.10}).
It is possible that one could use methods from \cite{Kr_Pi_So} and \cite{Ka_Mi_Si} to prove it (at least for a countable language). In the current paper, we will prove it in some special (but still rather general) situations.

Recall that the \emph{commutator length} of an element $g \in [G,G]$ is the minimal number of commutators sufficient to express $g$ as their product. The \emph{commutator width} $\cw(G)$ of $G$ is the maximum of the commutator lengths of elements of its derived subgroup $[G,G]$.  Notice that, by a compactness argument, the clause `$G^*$ is perfect' is equivalent to `$G$ is perfect and $G$ has finite commutator width $\cw(G)$'. By \cite[Theorem 3.5]{abscon}, every absolutely connected group is perfect.

Modifying the argument used in the proof of Proposition \ref{prop:x}(2), we obtain the following variant of Proposition \ref{prop:x}(2).




\begin{proposition}\label{prop:xx}
Consider Situation $(00)$.
\begin{enumerate} 
\item[(i)] Assume that $\Hom\left({G^*}^{000}_B,A_0/\left(A^*_1 \cap A_0\right)\right)$ is trivial. 
Then the induced 2-cocycle 
\[\overline{h^*}_{|{G^*}^{000}_B \times {G^*}^{000}_B} \colon {G^*}^{000}_B \times {G^*}^{000}_B \longrightarrow A_0/(A^*_1 \cap A_0)\] is non-split if and only if $\widetilde{G^*}^{000}_B \cap A^* \not\subseteq A^*_1$.

\item[(ii)] Assume that $G$ is absolutely connected of finite commutator width (and so ${G^*}^{000}_B={G^*}^{00}_B=G^*$ is perfect). Then the induced 2-cocycle $\overline{h^*}_{|{G^*}^{00}_B \times {G^*}^{00}_B} $ 
is non-split if and only if $\widetilde{G^*}^{00}_B \cap A^* \not\subseteq A^*_1$. These two equivalent conditions are also equivalent to the fact that the induced 2-cocycle $\overline{h^*} \colon G^* \times G^* \to A^*/A^*_1$ is non-split.
\end{enumerate}
\end{proposition}

The assumption that $\Hom\left({G^*}^{000}_B,A_0/\left(A^*_1 \cap A_0\right)\right)$ is trivial is satisfied in many cases. For instance, this assumption holds when ${G^*}^{000}_B$ is a perfect group or a divisible group, because $A_0/\left(A^*_1 \cap A_0\right)$ is a finitely generated abelian group. For example, we know that ${G^*}^{000}_B=G^*$ is perfect when $G$ is absolutely connected of finite commutator width, and ${G^*}^{000}_B$ is divisible when $G$ is a finitely generated abelian group \cite[Proposition 3.7]{bcg}.

\begin{proof} (i) $(\Rightarrow)$ follows immediately from $(\Rightarrow)$ in Proposition \ref{prop:x}(2). 

$(\Leftarrow)$. Suppose for a contradiction that $(\Leftarrow)$ does not hold. We get  $\overline{h^*}_{|{G^*}^{000}_B \times {G^*}^{000}_B}(x,y)=f(x)+f(y)-f(xy)$ for some function $f \colon {G^*}^{000}_B \to A_0/A^*_1$ satisfying $f(e)=0+A^*_1$. To apply $(\Leftarrow)$ from Proposition \ref{prop:x}(2), it is enough to check that $f$ is $B$-invariant. For this it is enough to show that $f$ is unique. So, suppose that $\overline{h}(x,y)=f_1(x)+f_1(y)-f_1(xy)$ for some function $f_1$.  Then $(f-f_1)(xy)=(f-f_1)(x)+(f-f_1)(y)$, so $f-f_1$ belongs to $\Hom\left({G^*}^{000}_B,A_0/A^*_1\right)$ which is trivial (this is the only place where this assumption is used). Thus, we get that $f=f_1$, so we are done.\\[3mm]
(ii) We have that ${G^*}^{000}_B={G^*}^{00}_B=G^*$ is perfect. The implication $(\Rightarrow)$ follows from the implication $(\Rightarrow)$ in point (i).

$(\Leftarrow)$. 
%
The idea is to apply the proof of $(\Leftarrow)$ in Proposition \ref{prop:x}(2), noticing that by Beth's definability theorem, the function $f$ considered in this proof is a $B$-type-definable subset of $G^* \times \left( A^*/A^*_1 \right)$ (i.e. its preimage in $G^* \times A^*$ under the map $(g,a) \mapsto (g,a+A^*_1)$ is type-definable over $B$).

Denote by $L$ the language of $\G$, and by $L_B$ its expansion by the constants from $B$.
Let $L_1 =L_B \cup \{ f_1 \}$ and $L_2 =L_B \cup\{ f_2 \}$, where $f_1$ and $f_2$ are two new distinct function symbols. Let $A(x)$ and $G(y)$ be formulas in $L$ defining $A$ and $G$ in $\G$, respectively. Let $A_1(x)$ be a type (in $L_B$) defining $A^*_1$ in $\G^*$. We will use the fact that $h$ is a function definable in $\G$ in the language $L_B$.

For $i\in \{1,2\}$ we define a theory $T_i$ in the language $L_i$ as the theory of $\G$ in $L_B$ together with the sentence \[(\forall x)\left( G(x) \rightarrow A(f_i(x)\right)\] and the following collection of formulas in $L_i$
\[(\forall x,y)\left[ (G(x) \wedge G(y)) \rightarrow \varphi(h(x,y)-f_i(xy)+f_i(x)+f_i(y))\right]\]
with $\varphi(z)$ ranging over all formulas from $A_1(z)$. 
This extra collection of formulas says that in any model $M$ of $T_i$, the induced 2-cocycle $\overline{h}^M \colon G(M) \times G(M) \to A(M)/A_1(M)$ coincides with the function $\overline{f_i}^M(xy)-\overline{f_i}^M(x)-\overline{f_i}^M(y)$, where $\overline{f_i}^M \colon G(M) \to A(M)/A_1(M)$ is given by $\overline{f_i}^M(x)=f_i^M(x) +A_1(M)$ (where $f_i^M$ is the interpretation of $f_i$ in $M$).

It follows easily that whenever $M$ is a model of $T_1 \cup T_2$, then the difference $\overline{f_1}^M-\overline{f_2}^M$ belongs to $\Hom\left( G(M), A(M)/A_1(M)\right)$, which is trivial because $G(M)$ is perfect and $A(M)/A_1(M)$ is abelian. So $\overline{f_1}^M=\overline{f_2}^M$.
Therefore, $T_1 \cup T_2 \models p_1(x,y) \equiv p_2(x,y)$, where $p_i(x,y)$ is the type 
\[G(x) \wedge A_1(y-f_i(x))\]
in the language $L_i$.

Using Beth's theorem (i.e. Fact \ref{fact:Beth for types}), we get a type $p(x,y)$ in $L_B$ such that 
\[T_1 \models p_1(x,y) \equiv p(x,y).\]

Suppose for a contradiction that $\overline{h^*} \colon G^*\times G^* \to A^*/A^*_1$ is split. By Remark \ref{rem:trivial action}(iii), this means that $\overline{h^*}(x,y)=f(x)+f(y)-f(xy)$ for some function $f \colon G^* \to A^*/A^*_1$. We can write $f(x)=f'(x)+A^*_1$ for some function $f' \colon G^* \to A^*$.

Expanding $\G^*$ by $f'$ (prolongated arbitrarily outside $G^*$) and treating $f'$ as the interpretation of the function symbol $f_1$, $\G^*$ becomes a model of $T_1$. We conclude that $f(x)=y+A^*_1$ if and only if
$\G^* \models p(x,y)$.

Having that $f$ is type-definable over $B$, we get that $\Phi$ (defined in the proof of \ref{prop:x}(2)) is also type-definable over $B$. Hence, modifying slightly the rest of the proof of \ref{prop:x}(2), one easily gets $\widetilde{G^*}^{00}_B \cap A^* \subseteq A^*_1$, which is a contradiction.
\end{proof}


Now, we prove Conjecture \ref{con:new 1} in two special cases.

\begin{proposition}\label{prop:xxx}
\begin{enumerate}
\item[(i)] Consider Situation $(00)$. Assume that $G$ is absolutely connected of finite commutator width. Then, $\widetilde{G^*}^{00}_B \cap A^*\subseteq A^*_1$ iff $\widetilde{G^*}^{000}_B \cap A^*\subseteq A^*_1$.
\item[(ii)] Consider Situation $(0)$. Let $A^*_1$ be a bounded index subgroup of $A^*$ which is invariant over $B$, invariant under the action of $G^*$, and which is an intersection of definable subgroups of finite index. Assume ${G^*}^{000}_B = {G^*}^{00}_B$. Then
$\widetilde{G^*}^{00}_B \cap A^*\subseteq A^*_1$ iff $\widetilde{G^*}^{000}_B \cap A^*\subseteq A^*_1$.
\end{enumerate}
\end{proposition}

\begin{proof}
(i) Only the implication $(\Leftarrow)$ requires an explanation. Assume $\widetilde{G^*}^{000}_B \cap A^*\subseteq A^*_1$. 
%
%
By Proposition \ref{prop:xx}(i), this implies that $\overline{h} \colon {G^*}^{000}_B \times {G^*}^{000}_B \to A_0/(A^*_1 \cap A_0)$ is split. Thus, since ${G^*}^{000}_B={G^*}^{00}_B$, Proposition \ref{prop:xx}(ii) gives us that $\widetilde{G^*}^{00}_B \cap A^* \subseteq A^*_1$.

(ii) Once again only the implication $(\Leftarrow)$ requires a proof. We start from the following claim.

\begin{claim}
The group $A^*_1$ can be written as $\bigcap_{i \in I} A_i$, where all $A_i$'s are $B$-definable (in $\G^*$) subgroups of $A^*$ of finite index, invariant under the action of $G^*$ and so normal in $\widetilde{G^*}$.
\end{claim}


\begin{proof}[Proof of the claim]
We can write $A^*_1= \bigcap_{i \in I} C_i$, where all $C_i$'s are definable subgroups of $A^*$ of finite index. Since $A^*_1\leq C_i$, $A^*_1$ is invariant under $G^*$, and $A^*/A^*_1$ is of bounded size, we get that the orbit of the set $C_i$ under $G^*$ is of bounded size, and so the setwise stabilizer $\Stab_{G^*}(C_i)$ of $C_i$ in $G^*$ is a definable subgroup of bounded index. Thus, $\left[G^*:\Stab_{G^*}(C_i)\right]<\aleph_0$. Define $B_i$ as the intersection of all $g\cdot C_i$ for $g$ ranging over $G^*$. We conclude that $A^*_1=\bigcap_{i \in I} B_i$, and all $B_i$'s are definable subgroups of $A^*$ of finite index, invariant under $G^*$.

Since $A^*_1\leq B_i$, $A^*_1$ is invariant over $B$, and $A^*/A^*_1$ is of bounded size, we get that the orbit of the definable group $B_i$ under $\aut(\G^*/B)$ is of bounded size, and so it is finite. Let $A_i$ be the intersection of all $f\left[B_i\right]$ for $f$ ranging over $\aut(\G^*/B)$. We conclude that $A^*_1=\bigcap_{i \in I} A_i$, and all $A_i$'s are $B$-definable subgroups of $A^*$ of finite index, invariant under $G^*$.
\end{proof}

The quotient $\widetilde{G^*}/A_i$ can and will be freely identified with the extension of $G^*$ by $A^*/A_i$ defined by means of the 2-cocycle induced by $h^*$ and the action of $G^*$ on $A^*/A_i$ induced by the action of $G^*$ on $A^*$. Then, $\pi\colon \widetilde{G^*}/A_i \to G^*$ denotes the projection on the second coordinate.

We see that $\left(\widetilde{G^*}/A_i\right)^{00}_B=\widetilde{G^*}^{00}_B/A_i$ and $\left(\widetilde{G^*}/A_i \right)^{000}_B=\widetilde{G^*}^{000}_B/A_i$. Since ${G^*}^{000}_B = {G^*}^{00}_B$, we conclude that
\[\pi \left[ \left(\widetilde{G^*}/A_i\right)^{00}_B \right]= {G^*}^{00}_B={G^*}^{000}_B=\pi \left[\left(\widetilde{G^*}/A_i\right)^{000}_B\right].\]
Therefore, $(A^*/A_i)  \left(\widetilde{G^*}/A_i\right)^{000}_B \geq \left(\widetilde{G^*}/A_i\right)^{00}_B$. Using the assumption that $A^*/A_i$ is finite, we conclude that 
\[\left[ \left(\widetilde{G^*}/A_i\right)^{00}_B : \left(\widetilde{G^*}/A_i\right)^{000}_B \right]<\aleph_0.\]
By virtue of \cite[Lemma 3.9]{gis}, this implies that $\left(\widetilde{G^*}/A_i\right)^{00}_B = \left(\widetilde{G^*}/A_i\right)^{000}_B$. Therefore, $\widetilde{G^*}^{00}_B \cap A^* \subseteq \left(\widetilde{G^*}^{000}_B \cap A^*\right) + A_i$. Since we have assumed that 
$\widetilde{G^*}^{000}_B \cap A^* \subseteq A^*_1 \subseteq A_i$, we conclude that $\widetilde{G^*}^{00}_B \cap A^* \subseteq A_i$. As this holds for any $i \in I$, we get $\widetilde{G^*}^{00}_B \cap A^* \subseteq A^*_1$.
\end{proof}


In the remaining part of  this section, we analyze the quotient $\widetilde{G^*}^{00}_B/\widetilde{G^*}^{000}_B$.

\begin{proposition}\label{prop:quotient}
Consider Situation $(0)$. Suppose that ${G^*}^{000}_B={G^*}^{00}_B$. Then $\widetilde{G^*}^{00}_B/\widetilde{G^*}^{000}_B$ is isomorphic to $\left(\widetilde{G^*}^{00}_B \cap A^*\right)/\left(\widetilde{G^*}^{000}_B \cap A^*\right)$, and so it is abelian. More precisely, there exists a $B$-invariant isomorphisms between these groups.
\end{proposition}
\begin{proof}
For each $g \in {G^*}^{000}_B = {G^*}^{00}_B$ choose $a_g \in A^*$ so that $(a_g,g) \in \widetilde{G^*}^{000}_B$. Using the exact sequence (\ref{eq:3}), one easily gets:
\[\begin{array}{ll}
\widetilde{G^*}^{000}_B \cap \left(A^* \times \{g\}\right) = \left(a_g + \left(\widetilde{G^*}^{000}_B \cap A^*\right)\right) \times \{ g \}, \\ \widetilde{G^*}^{00}_B \cap (A^* \times \{g\}) = \left(a_g + \left(\widetilde{G^*}^{00}_B \cap A^*\right)\right) \times \{ g \}.
\end{array}\]

We will show that the formula
\[\Phi\left((a,g)\widetilde{G^*}^{000}_B\right)= a-a_g +\left(\widetilde{G^*}^{000}_B \cap A^*\right)\]
is a well definition of a $B$-invariant isomorphism \[\Phi \colon \widetilde{G^*}^{00}_B/\widetilde{G^*}^{000}_B \to \left(\widetilde{G^*}^{00}_B \cap A^*\right)/\left(\widetilde{G^*}^{000}_B \cap A^*\right).\]


We check that $\Phi$ is well-defined. Consider any $(a_1,g_1), (a_2,g_2) \in \widetilde{G^*}^{00}_B$ such that \[ (a_1,g_1) (a_2,g_2)^{-1} \in \widetilde{G^*}^{000}_B.\] Our goal is to show that $(a_1-a_{g_1})-(a_2 -a_{g_2}) \in \widetilde{G^*}^{000}_B \cap A^*$. Since $(a_{g_1},g_1),(a_{g_2},g_2) \in \widetilde{G^*}^{000}_B$, we have: 
\[\begin{array}{ll}
\left(a_1-g_1g_2^{-1}\cdot a_2-g_1\cdot h\left(g_2^{-1},g_2\right) +h \left(g_1,g_2^{-1}\right),g_1g_2^{-1}\right)= (a_1,g_1) (a_2,g_2)^{-1} \in \widetilde{G^*}^{000}_B,\\
\left(a_{g_1}-g_1g_2^{-1}\cdot a_{g_2} -g_1\cdot h\left(g_2^{-1},g_2\right) +h \left(g_1,g_2^{-1}\right),g_1g_2^{-1}\right)=(a_{g_1},g_1)(a_{g_2},g_2)^{-1} \in \widetilde{G^*}^{000}_B.
\end{array}\]
This implies
\[(a_1-a_{g_1})-g_1g_2^{-1}\cdot (a_2-a_{g_2}) \in \widetilde{G^*}^{000}_B \cap A^*.\]
On the other hand, by Remark \ref{rem:trivial action}(ii), 
\[g_1g_2^{-1}\cdot (a_2-a_{g_2}) -(a_2-a_{g_2}) \in \widetilde{G^*}^{000}_B \cap A^*.\]
So, we conclude that $(a_1-a_{g_1})-(a_2 -a_{g_2}) \in \widetilde{G^*}^{000}_B \cap A^*$.

It is easy to see that $\Phi$ is invariant over $B$ and surjective. Hence, it remains to check that it is an injective homomorphism.

First, we check that $\Phi$ is a homomorphism. Take any $(a_1,g_1),(a_2,g_2) \in \widetilde{G^*}^{00}_B$. Then,
\[\Phi\left((a_1,g_1)(a_2,g_2)\widetilde{G^*}^{000}_B\right)= a_1+g_1\cdot a_2 +h(g_1,g_2)- a_{g_1g_2} +\left(\widetilde{G^*}^{000}_B \cap A^*\right)\] and \[\Phi\left((a_1,g_1)\widetilde{G^*}^{000}_B\right)+\Phi\left((a_2,g_2) \widetilde{G^*}^{000}_B\right) = a_1-a_{g_1}+ a_2- a_{g_2} + \left(\widetilde{G^*}^{000}_B \cap A^*\right).\] So, our goal is to show that
\[\left(a_{g_1}+g_1\cdot a_2 +h(g_1,g_2)-a_{g_1g_2}\right) - (a_2-a_{g_2}) \in \left(\widetilde{G^*}^{000}_B \cap A^*\right).\]

Since $(a_{g_1}+g_1\cdot a_{g_2} +h (g_1,g_2),g_1g_2)=(a_{g_1},g_1)(a_{g_2},g_2) \in \widetilde{G^*}^{000}_B$, we have that $a_{g_1}+g_1\cdot a_{g_2} +h(g_1,g_2) -a_{g_1g_2} \in \widetilde{G^*}^{000}_B \cap A^*$. Therefore, 
\[\begin{array}{lll}
&\left(a_{g_1}+g_1\cdot a_2 +h(g_1,g_2)-a_{g_1g_2}\right) -(a_2-a_{g_2})
= \left(a_{g_1}+g_1\cdot a_{g_2} +h(g_1,g_2) -a_{g_1g_2}\right) \\+& g_1\cdot (a_2-a_{g_2}) -(a_2-a_{g_2}) 
\in g_1\cdot (a_2-a_{g_2}) - (a_2-a_{g_2})+ \left(\widetilde{G^*}^{000}_B \cap A^*\right).
\end{array}\]
We are done, because Remark \ref{rem:trivial action}(ii) implies that $ g_1\cdot (a_2-a_{g_2}) - (a_2-a_{g_2}) \in \widetilde{G^*}^{000}_B \cap A^*$.

It remains to show that $\Phi$ is injective. Consider any $(a_1,g_1),(a_2,g_2) \in \widetilde{G^*}^{00}_B$ such that 
\[(a_1,g_1)(a_2,g_2)^{-1} \notin \widetilde{G^*}^{000}_B.\] Then,
\begin{equation}\
a_1- g_1g_2^{-1}\cdot a_2 -g_1\cdot h\left(g_2^{-1},g_2\right) + h\left(g_1,g_2^{-1}\right) -a_{g_1g_2^{-1}} \notin \widetilde{G^*}^{000}_B \cap A^*.\label{equation:(7)}
\end{equation}

Our goal is to show that $(a_1-a_{g_1})-(a_2-a_{g_2}) \notin \widetilde{G^*}^{000}_B \cap A^*$. Suppose for a contradiction that this is not the case.

As $(a_{g_1},g_1),(a_{g_2},g_2) \in \widetilde{G^*}^{000}_B$, we have that
\[\left(a_{g_1}-g_1g_2^{-1}\cdot a_{g_2} - g_1\cdot h\left(g_2^{-1},g_2\right) +h\left(g_1,g_2^{-1}\right),g_1g_2^{-1}\right) = (a_{g_1},g_1)(a_{g_2},g_2)^{-1} \in \widetilde{G^*}^{000}_B,\]
and so 
\[a_{g_1}-g_1g_2^{-1}\cdot a_{g_2} -g_1\cdot h\left(g_2^{-1},g_2\right) +h\left(g_1,g_2^{-1}\right)-a_{g_1g_2^{-1}} \in\widetilde{G^*}^{000}_B \cap A^*.\]

Using this together with the assumption that $(a_1-a_{g_1})-(a_2-a_{g_2}) \in \widetilde{G^*}^{000}_B \cap A^*$ and the fact $g_1g_2^{-1}\cdot (a_{g_2}-a_2)-(a_{g_2}-a_2) \in \widetilde{G^*}^{000}_B \cap A^*$ (which follows from Remark \ref{rem:trivial action}(ii)), we get the following computation, which contradicts (\ref{equation:(7)}):
\[\begin{array}{lll}
& a_1- g_1g_2^{-1}\cdot a_2 -g_1\cdot h\left(g_2^{-1},g_2\right) + h\left(g_1,g_2^{-1}\right) -a_{g_1g_2^{-1}} \\
&= \left(a_{g_1}-g_1g_2^{-1}\cdot a_{g_2} - g_1 \cdot h\left(g_2^{-1},g_2\right) + h\left(g_1,g_2^{-1}\right)-a_{g_1g_2^{-1}}\right) \\
&+ (a_1-a_{g_1})+g_1g_2^{-1}\cdot (a_{g_2}-a_2)
 \in \widetilde{G^*}^{000}_B \cap A^*.
\end{array}\]
\end{proof}

In the examples from \cite{conv_pillay}, and, more generally, for each group $G$ definable in a monster model of an $o$-minimal expansion of a real closed field, the quotient $G^{00}/G^{000}$ turns out to be abstractly isomorphic to an abelian, compact Lie group divided by a dense, finitely generated subgroup.

In Situation $(00)$, assuming that ${G^*}^{000}_B = {G^*}^{00}_B$ and $A^*_1 \subseteq \widetilde{G^*}^{000}_B \cap A^*$, and using Proposition \ref{prop:quotient}, we get 
\[\widetilde{G^*}^{00}_B/\widetilde{G^*}^{000}_B \cong \left( \left(\widetilde{G^*}^{00}_B \cap A^*\right)/A^*_1\right)/\left( \left( \widetilde{G^*}^{000}_B \cap A^*\right)/A^*_1\right).\]
%
Thus, since by Claim 1 in the proof of Theorem \ref{thm:main} we know that $\widetilde{G^*}^{000}_B \cap A^* \leq A^*_1 +A_0$, we easily conclude that $\widetilde{G^*}^{00}_B/\widetilde{G^*}^{000}_B$ is isomorphic to the quotient of an abelian, compact group by a finitely generated subgroup. A question arises if this finitely generated subgroup is dense. Proposition \ref{prop:xxx} yields the following observation.

\begin{proposition}\label{prop:density}
\begin{enumerate}
\item[(i)]  Consider Situation $(00)$. Assume $G$ is absolutely connected of finite commutator width.  Then $\left( \widetilde{G^*}^{000}_B \cap A^*\right)/A^*_1$ is dense in $\left(\widetilde{G^*}^{00}_B \cap A^*\right)/A^*_1$.
\item[(ii)] Consider Situation $(0)$. Let $A^*_1$ be a bounded index subgroup of $A^*$ which is invariant over $B$, invariant under the action of $G^*$, and which is an intersection of definable subgroups of finite index.
Assume ${G^*}^{000}_B = {G^*}^{00}_B$. Then $\left( \widetilde{G^*}^{000}_B \cap A^*\right)/A^*_1$ is dense in $\left(\widetilde{G^*}^{00}_B \cap A^*\right)/A^*_1$.
\end{enumerate}
\end{proposition}

\begin{proof}
(i) Let $\rho \colon A^* \to A^*/A^*_1$ be the quotient map, and let $A^*_2$ be the preimage under $\rho$ of the closure of $\left( \widetilde{G^*}^{000}_B \cap A^*\right)/A^*_1$. We see that $A^*_2$ is a bounded index, $B$-type-definable subgroup of $A^*$ invariant under the action of $G^*$. Moreover,  $\widetilde{G^*}^{000}_B \cap A^* \subseteq A^*_2$. By Proposition \ref{prop:xxx}(i) applied to $A^*_2$ (instead of $A^*_1$), we get that $\widetilde{G^*}^{00}_B \cap A^* \subseteq A^*_2$, which implies that $\left( \widetilde{G^*}^{000}_B \cap A^*\right)/A^*_1$ is dense in $\left(\widetilde{G^*}^{00}_B \cap A^*\right)/A^*_1$.

(ii) By the claim from the proof of Proposition \ref{prop:xxx}, we know that $A^*_1=\bigcap_{i \in I} A_i$, where all $A_i$'s are $B$-definable subgroups of $A^*$ of finite index, invariant under $G^*$. We can assume that the family $\{ A_i: i \in I\}$ is closed under finite intersections. Take the notation from the proof of (i). Since $A^*_2$ is $B$-type-definable and contains $A^*_1$, we get that $A^*_2=\bigcap_{i \in I}(A_i + A^*_2)$ and each $A_i +A^*_2$ is a $B$-definable subgroup of $A^*$ of finite index, invariant under $G^*$. Moreover, $\widetilde{G^*}^{000}_B \cap A^* \subseteq A^*_2$. Thus, applying Proposition \ref{prop:xxx}(ii), we get that $\widetilde{G^*}^{00}_B \cap A^* \subseteq A^*_2$, which is enough. 
\end{proof}

W finish this section with a remark which follows from the proof of Proposition \ref{prop:density}. 
\begin{remark}\label{rem:zastosowanie 2.10}
Consider Situation $(00)$, and assume that ${G^*}^{000}_B = {G^*}^{00}_B$ and $A^*_1 \subseteq \widetilde{G^*}^{000}_B \cap A^*$. If Conjecture \ref{con:new 1} is true, then $\left( \widetilde{G^*}^{000}_B \cap A^*\right)/A^*_1$ is dense in $\left(\widetilde{G^*}^{00}_B \cap A^*\right)/A^*_1$.
\end{remark}


\section{Extensions} \label{sec:finite}

In this section, we make a few observations telling us that in some situations in which Theorem \ref{thm:main}, Corollary \ref{cor:main} or Corollary \ref{cor:main2} can be applied to a group $G$, it can also be applied to certain extensions of $G$ (providing new examples in Section \ref{sec:central}). 

By the \emph{strong version of Assumption (i)} of Theorem \ref{thm:main}, we mean the version in which we assume non-splitness with respect to all functions instead of $B$-invariant functions.

Suppose $G$ acts on an abelian group $A$, $h\colon G\times G\to A$ is a 2-cocycle and $f\colon H \twoheadrightarrow G$ is an epimorphism. Then $f$ induces an action of $H$ on $A$ so that the composition $h':= h\circ (f,f)\colon H \times H \to A$ becomes a 2-cocycle on $H$. Denote by $z\colon G \to H$ a section of $f$ with $z(e_G)=e_H$. By using the classical idea of the `Inflation-Restriction Sequence' \cite[Theorem 4.1.20]{rosen}, we obtain the following proposition.

\begin{proposition} \label{rem:finite}
Under the notation of the previous paragraph, the 2-cocycle $h':=h\circ (f,f)\colon H\times H\to A$ is split if and only if there exists a homomorphism $g\colon \ker(f) \to A$ which is $G$-invariant, that is for $x'\in H$ and $y'\in\ker(f)$ we have $g\left(x'y'x'^{-1}\right)=f(x')\cdot g(y')$, and such that $h$ is cohomologous to the 2-cocycle $h''$ defined by \[h''(x,y) = g\left(z(xy)z(y)^{-1}z(x)^{-1}\right)\] for $x,y\in G$.

In particular, if $\Hom(\ker(f),A)$ is trivial, then $h$ is non-split if and only if $h'$ is non-split.
\end{proposition}

\begin{proof}
$(\Rightarrow)$. There exists a function $g\colon H \to A$ such that for $x',y'\in H$, \[h(f(x'),f(y')) = g(x')+f(x')\cdot g(y')-g(x'y'). 
\] 
If $x'\in \ker(f)$ or $y'\in \ker(f)$, then $g(x'y')=g(x') + f(x')\cdot g(y')$ (because $h(e_G,-)=h(-,e_G)=0$), so $g_{| \ker(f)} \in\Hom(\ker(f),A)$. 

Moreover, taking $x'\in H$ and $y'\in\ker(f)$, we have $x'y'x'^{-1}\in\ker(f)$, so $g(x'y')=g(x'y'x'^{-1}x')=g(x'y'x'^{-1})+g(x')$, and therefore $g(x'y'x'^{-1})=g(x'y') - g(x')=f(x')\cdot g(y')$. This means that $g_{| \ker(f)}$ is $G$-invariant.

Denote $\widetilde{z} = g\circ z\colon G \to A$. For $x,y\in G$ let $x'=z(x)$ and $y'=z(y)$. We have
\begin{eqnarray*}
h(x,y) = h(f(x'),f(y')) &=& g(z(x)) + x\cdot g(z(y)) - g(z(x)z(y)) \\
&=& \left[ g(z(x)) + x\cdot g(z(y)) - g(z(xy))\right] + g(z(xy)) - g(z(x)z(y)) \\
&=& \left[ \widetilde{z}(x) + x\cdot \widetilde{z}(y) - \widetilde{z}(xy)\right] + g\left(z(xy)z(y)^{-1}z(x)^{-1}\right),
\end{eqnarray*}
because $z(xy)z(y)^{-1}z(x)^{-1}\in\ker(f)$.

$(\Leftarrow)$. If $h$ is cohomologous to $h''$, then $h'=h\circ (f,f)$ is cohomologous to $h''':=h''\circ (f,f)$. For $x',y'\in H$ we have 
\begin{eqnarray*}
h'''(x',y') &=& g\left(z(f(x'y'))z(f(y'))^{-1}z(f(x'))^{-1}\right) \\
 &=& g\left((x'y'z(f(x'y'))^{-1})^{-1}x'(y'z(f(y'))^{-1})x'^{-1} (x'z(f(x'))^{-1})\right) \\
  &=& F(x')+f(x')\cdot F(y')-F(x'y'),
\end{eqnarray*}
where $F\colon H\to A$ is defined as $F(t)=g(tz(f(t))^{-1})$. Hence, both $h'''$ and $h'$ are split.
\end{proof}


\begin{corollary} \label{cor:ext}
Consider Situation $(00)$. Suppose that Assumption (ii) and the strong version of Assumption (i) of Theorem \ref{thm:main} hold. Assume that 
\begin{enumerate}
\item $G$ is absolutely connected,
\item $f\colon H \twoheadrightarrow G$ is an epimorphism.
\end{enumerate}
Regard $H$ and $f$ as objects $\emptyset$-definable in some first order expansion $\H$ of $\G$. Let $\widetilde{H}$ be the extension of $H$ by $A$ corresponding to the 2-cocycle $h':=h\circ (f,f)$.
%
%
Let $\H^*\succ \H$ be a big enough monster model. Assume additionally that \[\Hom\left(\ker\left({f^*}_{|{H^*}^{00}_B}\right),A_0/(A^*_1 \cap A_0)\right)\] is trivial, where $f^*$ in the interpretation of $f$ in $\H^*$. Then the 2-cocycle  \[\overline{h'^*}_{|{H^*}^{00}_B \times {H^*}^{00}_B}=\overline{h^*}\circ(f^*,f^*)_{|{H^*}^{00}_B \times {H^*}^{00}_B}\colon {H^*}^{00}_B \times {H^*}^{00}_B \longrightarrow A_0/(A^*_1 \cap A_0),\] where $\overline{h^*}\colon G^* \times G^* \to A_0/(A^*_1\cap A_0)$ is the induced 2-cocycle, is non-split. Therefore, by Theorem \ref{thm:main}, $\widetilde{H^*}^{000}_B \ne \widetilde{H^*}^{00}_B$.
\end{corollary}

\begin{proof}
Choose a monster model $\H^* \succ \H$ so that the interpretation of $\G$ in it is a monster model of $\Th(\G)$ being an elementary extension of $\G^*$; we may assume that this interpretation coincides with $\G^*$.

Since $G$ is absolutely connected, we have that ${G^*}^{00}_B$ (computed in the sense of $\H^*$) equals $G^*$, and so $f^*\left[{H^*}^{00}_B\right] =G^*$.
If $\overline{h'^*}_{|{H^*}^{00}_B \times {H^*}^{00}_B}$ is split, then by Proposition \ref{rem:finite} and the triviality of $\Hom\left(\ker\left({f^*}_{|{H^*}^{00}_B}\right),A_0/(A^*_1 \cap A_0)\right)$, the 2-cocycle $\overline{h^*}\colon G^* \times G^* \to A_0/(A^*_1 \cap A_0)$ is split, which contradicts the strong version of Assumption (i) of Theorem \ref{thm:main}.
\end{proof}

The extra assumption that $\Hom\left(\ker\left({f^*}_{|{H^*}^{00}_B}\right),A_0/(A^*_1 \cap A_0)\right)$ is trivial means that \[\Hom\left(\ker\left({f^*}_{|{H^*}^{00}_B}\right), \Z^n\right)\] is trivial, which is always satisfied for example when $\ker(f^*)$ is divisible or finite (because then $\ker\left({f^*}_{|{H^*}^{00}_B}\right)$ is also divisible or finite).

Corollary \ref{cor:ext} is a general recipe for obtaining new examples of extensions to which Theorem \ref{thm:main} can be applied. The next remark is a variant of this, where we assume that $H$ is a product extension of $G$, but the assumptions that `$G$ is absolutely connected' and `$\Hom\left(\ker\left({f^*}_{|{H^*}^{00}_B}\right),A_0/(A^*_1 \cap A_0)\right)$ is trivial' are dropped.

\begin{remark}\label{rem:produkty}
Consider Situation $(00)$. Suppose that Assumption (ii) and the strong version of Assumption (i) of Theorem \ref{thm:main} hold. Let $H=K \times G$, where $K$ is an arbitrary group. Let $\H$ be the expansion of $\G$ obtained by adding a new sort, consisting of the pure group structure $H$ and the projection $f \colon H \to G$ on the second coordinate.  Let $\widetilde{H}$ be the extension of $H$ by $A$ corresponding to the 2-cocycle $h':=h\circ (f,f)$. As usual, $\H^* \succ \H$ denotes a big enough monster model. Then the 2-cocycle \[\overline{h'^*}_{|{H^*}^{00}_B \times {H^*}^{00}_B}=\overline{h}\circ(f^*,f^*)_{|{H^*}^{00}_B \times {H^*}^{00}_B}\colon {H^*}^{00}_B \times {H^*}^{00}_B \longrightarrow A_0/(A^*_1 \cap A_0),\] 
is non-split. Therefore, by Theorem \ref{thm:main}, $\widetilde{H^*}^{000}_B \ne \widetilde{H^*}^{00}_B$.
\end{remark}

\begin{proof}
Let $\H_1$ be the expansion of $\G$ by the additional sort for the pure group structure $K$; then the group $H$ and the projection $f$ are definable in $\H_1$. Take a monster model $\H_1^* \succ \H_1$ so big that the interpretation of $\H$ in $\H_1^*$ (which we denote by $\H^*$) is a monster model of $\Th (\H)$ and the interpretation of $\G$ in $\H_1^*$ is a monster model of $\Th (\G)$ being an elementary extension of $\G^*$; we may assume that the interpretation of $\G$ in $\H_1^*$ coincides with $\G^*$. Let $K^*$ and $H^*$ be the interpretations of $K$ and $H$ in $\H_1^*$.

Then $H^*=K^* \times G^*$. Since $h'((k_1,g_1),(k_2,g_2))=h(g_1,g_2)$ and $\overline{h^*}_{|{G^*}^{00}_B \times {G^*}^{00}_B} \colon {G^*}^{00}_B \times {G^*}^{00}_B \to A_0/(A^*_1 \cap A_0)$ is non-split, in order to show that $\overline{h'^*}_{|{H^*}^{00}_B \times {H^*}^{00}_B} \colon {H^*}^{00}_B \times {H^*}^{00}_B \to A_0/(A^*_1 \cap A_0)$ is non-split, it is enough to note that ${G^*}^{00}_B$ computed in the sense of $\H_1^*$ coincides with ${G^*}^{00}_B$ computed in the sense of $\G^*$. But this follows immediately from an easy observation that the structure induced on $\G^*$ from $\H_1^*$ coincides with the original structure on $\G^*$.
\end{proof}


Part $(2)$ of the next proposition says that any strongly non-split extension of an absolutely connected group by a finite abelian group is also absolutely connected. 

\begin{proposition} \label{prop:ext}
Let $A' \to H \to G$
be an extension of a group $G$ by a finite abelian group $A'$, defined by means of a 2-cocycle $h_f\colon G\times G\to A'$. Assume $h_f$ is strongly non-split in the following way: for every proper subgroup $A''\lneqq A'$ normal in $H$ (equivalently, invariant under the action of $G$) the induced 2-cocycle $\overline{h_f}\colon G\times G\to A'/A''$ is non-split. Then:
\begin{enumerate}
\item $G$ has no proper subgroups of finite index if and only if $H$ has no proper subgroups of finite index.
\item $G$ is absolutely connected if and only if $H$ is absolutely connected. 
%
%
\end{enumerate}
\end{proposition}
\begin{proof}
$(1)$ The implication $(\Leftarrow)$ is obvious. We prove $(\Rightarrow)$. Suppose $H_1 \leq H$ is a subgroup of finite index. We may assume that $H_1$ is a normal subgroup of $H$. Then $f[H_1]=G$, so $H=H_1\cdot A'$. We will show that $H=H_1$. Suppose for a contradiction that $A''=H_1\cap A'$ is a proper subgroup of $A'$. Consider the induced epimorphism $\overline{f}\colon H/A''\to G$. Since $A'=\ker(f)$, $\overline{f}_{|H_1/A''}$ is an isomorphism onto $G$. Hence, the inverse isomorphism $\left(\overline{f}_{|H_1/A''}\right)^{-1} \colon G \to H_1/A''$ is a section of $\overline{f}$. Therefore, the exact sequence
\[\xymatrix{
 1 \ar@{^{(}->}[r] & A'/A'' \ar@{^{(}->}[r] & H/A'' \ar@{>>}[r]^-{\overline{f}} & G \ar@{>>}[r] & 1}\] is split, and so the 2-cocycle $\overline{h_f}\colon G\times G\to A'/A''$ is also split, a contradiction.

$(2)$ By \cite[Proposition 2.8(4)]{abscon}, it is enough to prove that $H$ has no proper subgroups of finite index, which follows by $(1)$.
%
\end{proof}

The next corollary follows from Proposition \ref{rem:finite} and Proposition \ref{prop:ext}. 

\begin{corollary} \label{cor:ext_fin}
Consider the situation described in the first part of Corollary \ref{cor:main}. Suppose that 
\begin{enumerate}
\item $G$ is absolutely connected, 
\item $f\colon H \twoheadrightarrow G$ is an epimorphism with finite and abelian $\ker(f)$, and a 2-cocycle $h_f\colon G\times G\to \ker(f)$ corresponding to the extension $f$ is strongly non-split in the sense of Proposition \ref{prop:ext}.
\end{enumerate}
Regard $H$ and $f$ as objects $\emptyset$-definable in some first order expansion $\H$ of $\G$. Let $\widetilde{H}$ be the extension of $H$ corresponding to the 2-cocycle $h'=h\circ (f,f) \colon H \times H \to A_0$. Let $\H^* \succ \H$ be a  monster model. Then $H$ is absolutely connected, and the 2-cocycle $h'$ is non-split. Therefore, the assumptions of the first part of Corollary \ref{cor:main} are satisfied, so $\widetilde{H^*}^{000}_B \ne \widetilde{H^*}^{00}_B$. 
%
%
\end{corollary}

Example \ref{ex:counter-example} from Section \ref{sec:central} shows that even if in Corollary \ref{cor:ext_fin} we start from a situation in which all the assumptions of both parts of Corollary \ref{cor:main} are satisfied, it may happen that $\widetilde{H^*}^{00}_B \ne \widetilde{H^*}$.

Notice that letters $H$ and $f$ in Corollary \ref{cor:main2} denote different objects than usually in this section. 
In the next proposition, we keep the notation from Corollary \ref{cor:main2}, but we change the notation that we have been using in this section: the group which has been denoted in this section by $H$ will be denoted by $K$ and the function which has been denoted by $f$ will be denoted by $F$. In the next proposition, we consider the \emph{strong version of Assumption (5)} of Corollary \ref{cor:main2}, i.e., the version in which we use all (not necessarily $B$-definable) thick subsets of the group in question. 

\begin{proposition}\label{prop:more examples in 4.2}
Consider Situation $(00)$. Suppose that Assumptions (1), (2) and the strong version of Assumption (5) of Corollary \ref{cor:main2} hold. Let $F\colon K \twoheadrightarrow G$ be any epimorphism. Regard $K$ and $F$ as objects $\emptyset$-definable in some first order expansion $\Kcal$ of $\G$. Put $h'=h\circ (F,F) \colon K \times K \to A_0$ and $f'=f \circ F_{|F^{-1}[H]} \colon F^{-1}[H] \to A$. Let $\widetilde{K}$ be the extension of $K$ corresponding to the 2-cocycle $h'$. Let $\Kcal^* \succ \Kcal$ be a monster model. Then $f'$ witnesses that Assumptions (1), (2) and (5) of Corollary \ref{cor:main2} are satisfied for $A$, $B$, $A_0$, $A^*_1$, $K$ playing the role of $G$, $h'$ playing the role of $h$, and $F^{-1}[H]$ playing the role of $H$. Therefore,  $\widetilde{K^*}^{000}_B \ne \widetilde{K^*}^{00}_B$.
\end{proposition}



\section{Examples}\label{sec:central}

The aim of this section is to find new (concrete) examples of groups $G$ together with a 2-cocycle $h\colon G \times G \to \Z$ for which Theorem \ref{thm:main} or Corollaries \ref{cor:main} or \ref{cor:main2} can be applied, yielding the groups $\widetilde{G^*}$ satisfying $\widetilde{G^*}^{000}_B \ne \widetilde{G^*}^{00}_B$. 


We have divided this section into two subsections. In the first one, we use Matsumoto-Moore theory in order to achieve situations form Corollary \ref{cor:main}, and, in consequence, we obtain examples generalizing Example \ref{ex:PC1}; then we use results from Section \ref{sec:finite} to obtain yet more examples. In the second subsection, we use various quasi-characters to obtain split 2-cocycles
to which Corollary \ref{cor:main2} can be applied, yielding various examples; we also prove that the extension of $\SL_2(\Z)$ by $\Z$ defined by means of the 2-cocycle considered in Example \ref{ex:PC1} is an example where the two connected components are different. 
This group is a well-known braid group on 3 strands (see \cite[Section 7, Thm 18]{rawn} or \cite[Chapter 10, Thm 10.5]{milnor}). 



\subsection{Applications of Corollary \ref{cor:main}: central extensions of symplectic groups}\label{subsection:4.1}



We mainly concentrate on the case when $G=\Sp_{2n}(k)$ is the symplectic group for various infinite fields $k$. Recall that $\Sp_{2n}(k)$ is defined as $\left\{M\in\GL_{2n}(k) : M^{t}JM = J\right\}$, where $J = \left( \begin{array}{cc}
0 & \id_n \\
-\id_n & 0 \\ \end{array}\right)$, so, in particular, $\Sp_2(k)=\SL_2(k)$.

The main problem is to find 2-cocycles with finite image. We will use Matsumoto-Moore theory of the abstract universal central extension of $\Sp_{2n}(k)$. Every central extension of $\Sp_{2n}(k)$ by an abelian group $A$ corresponds to a 2-cocycle $\Sp_{2n}(k)\times \SL_{2n}(k) \to A$. However, in this case 
(more generally, in the case of \emph{Chevalley groups}), one can use a more convenient approach via \emph{Steinberg symbols} \cite[Section 7]{stein}, \cite{rehm}, which are certain mappings $c\colon k^{\times}\times k^{\times} \to A$.


For example, the topological universal cover $\widetilde{\SL_2(\R)}$ of $\SL_2(\R)$ can be described by the symbol defined by \cite[p. 51]{matsumoto}: for $x,y\in\R^{\times}$ \[c(x,y) = \left\{
  \begin{array}{r l}
    1 & \quad \text{if } x < 0 \text{ and } y<0 \\
    0 & \quad \text{otherwise}\\
  \end{array} \right.. \tag{\textasteriskcentered}\]

\begin{definition}{\cite[p. 74]{stein}} \label{def:uce}
A central extension $\pi\colon \widetilde{G} \to G$ is called \emph{universal} if for any central extension $\pi'\colon G' \to G$ there exists a unique homomorphism $f\colon \widetilde{G} \to G'$ such that $\pi'\circ f=\pi$, that is the following diagram commutes
\[\xymatrix{
 {}\widetilde{G} \ar@{>>}[r]^-{\pi} \ar[d]_-{f} & G \\
 G' \ar@{>>}[ru]_-{\pi'} &
}.\]
\end{definition}

It is known that any perfect group possesses a universal central extension \cite[p. 75]{stein}, which is unique up to isomorphism over $G$ \cite[p. 74]{stein}.

Suppose $k$ is an arbitrary infinite field, and let $k^{\times}=k\setminus\{0\}$ be the multiplicative group of $k$. 
%
%
Following \cite{rehm}, we describe the universal central extension of the symplectic group $\Sp_{2n}(k)$ (see also \cite[Sections 6, 7]{stein}). 
Since $\Sp_{2n}(k)$ is perfect, it has the universal central extension 
\[\xymatrix{
 1 \ar@{^{(}->}[r] & \ker(\pi)= \K2(k) \ar@{^{(}->}[rr] & & \St(k) \ar@{>>}[rr]^-{\pi} & &  \Sp_{2n}(k) \ar@{>>}[r] & 1.
} \tag{\textasteriskcentered\textasteriskcentered}\]

We recall the presentations of $\Sp_{2n}(k)$, $\St(k)$ and $\K2(k)$. Let $\Sigma = C_n$ be the root system for $\Sp_{2n}(k)$. Consider symbols $u_{\alpha}(x)$ for $\alpha\in\Sigma$ and $x\in k$. 
For $x\in k^{\times}$ define: $w_{\alpha}(x) = u_{\alpha}(x)u_{-\alpha}\left(-x^{-1}\right)u_{\alpha}(x)$ and $h_{\alpha}(x) = w_{\alpha}(x) w_{\alpha}(-1)$.

For $x,y\in k$, $t\in k^{\times}$, $\alpha, \beta\in\Sigma$ define:

\begin{itemize}
\item[(A)] $u_{\alpha}(x+y) = u_{\alpha}(x) u_{\alpha}(y)$,

\item[(B)] $\left[u_{\alpha}(x),u_{\beta}(y)\right] = \prod_{i,j>0,i\alpha+j\beta\in\Sigma}u_{i\alpha+j\beta}\left(c_{\alpha\beta,i,j} x^i y^j\right)$, for some constants $c_{\alpha\beta,i,j}\in\Z$, dependent only on $\Sigma$,

\item[(B$^\prime$)] $w_{\alpha}(t)u_{\alpha}(x)w_{\alpha}(t)^{-1} = u_{-\alpha}\left(-t^{-2}x\right)$,

\item[(C)] $h_{\alpha}(xy) = h_{\alpha}(x)h_{\alpha}(y)$.
\end{itemize}

A presentation of $\Sp_{2n}(k)$ is given by

\[\left\{
  \begin{array}{l l}
  \left\langle\  u_{\alpha}(x)\  |\  (A), (B'), (C)\ \right\rangle_{\alpha\in\Sigma, x\in k} &: \text{if }n = 1, \\
    \left\langle\ u_{\alpha}(x)\  |\  (A), (B), (C)\ \right\rangle_{\alpha\in\Sigma, x\in k}  &: \text{if }n\geq 2. \\
\end{array}
\right.
\]


We identify each $u_\alpha(x)$ with a concrete matrix from $\Sp_{2n}(k)$. Recall \cite[Page 205]{bour} that the roots for $C_n$ are $\Sigma = \{\pm 2e_i,\ \ \pm e_j \pm e_k : 1\leq i \leq n,\ \ 1\leq j < k \leq n\}\subset \R^n$, where $\{e_i : 1\leq i\leq n\}$ is the standard basis of $\R^n$. Each $2e_i$ is a \emph{long root}. Other roots from $\Sigma$ are \emph{short}.

\begin{remark} \label{rem:opis_mac}
Let $\{b_{-n},\ldots,b_{-1},b_1,\ldots,b_n\}$ be the standard basis of the vector space $k^{2n}$. Define $E_{i,j}=(a_{pq})$ as a $2n\times 2n$ matrix such that $a_{ij}=1$ and $a_{pq}=0$ for $(p,q)\neq(i,j)$. Put  \cite[Page 205]{bour}: $X_{2e_i} = E_{i,-i}$, $X_{-2e_i} = -E_{-i,i}$, $X_{e_i-e_j}=E_{i,j}-E_{-j,-i}$, $X_{-e_i+e_j}=-E_{j,i}+E_{-i,-j}$, $X_{e_i+e_j}=E_{i,-j}+E_{j,-i}$, $X_{-e_i-e_j}=-E_{-i,j}-E_{-j,i}$. Then for $\alpha\in \Sigma$ and $x\in k$ we have (see \cite[Section 3, Page 21]{stein}, where the symbol $x_{\alpha}(t)$ for $u_{\alpha}(t)$ is used) \[u_\alpha(x) = \exp(c\cdot X_\alpha) = \id_{2n} + x\cdot X_{\alpha},\] as $X_\alpha^2=0$.
\end{remark}

The Steinberg presentation of $\St(k)$ is 

\[\left\{
  \begin{array}{l l}
  \left\langle\  u_{\alpha}(x)\  |\  (A), (B')\ \right\rangle_{\alpha\in\Sigma, x\in k} &: \text{if }n = 1, \\
    \left\langle\ u_{\alpha}(x)\  |\  (A), (B)\ \right\rangle_{\alpha\in\Sigma, x\in k}  &: \text{if }n\geq 2. \\
\end{array}
\right.
\]


For a long root $\alpha\in\Sigma$ and  $x,y\in k^{\times}$ define $c(x,y) = h_{\alpha}(x)h_{\alpha}(y)h_{\alpha}(xy)^{-1}$. Matsumoto-Moore theorem (see \cite[Prop. 5.5, Thm. 5.10, Cor. 5.11]{matsumoto}, \cite[Theorem 12, p. 86]{stein}) tells us that $\K2(k)$ is central in $\St(k)$, it is generated by  $\left\{ c(x,y):x,y\in k^{\times}\right\} $ and can be presented abstractly as $\left\langle\ c(x,y) \ |\ (\text{S1}), (\text{S2}), (\text{S3})\ \right\rangle_{x,y\in k^{\times}}$,  where:
\begin{itemize}
\item[(S1)] $c(x,y)\; c(xy,z) = c(x,yz)\; c(y,z)$,
\item[(S2)] $c(1,1) = 1,\ c(x,y) = c(x^{-1},y^{-1})$,
\item[(S3)] $c(x,y) = c(x,(1-x)y)$ for $x\ne 1$.
\end{itemize}
We call each $c(x,y)$ a \emph{standard generator} of $\K2(k)$. Since $\K2(k)$ is abelian, from now on, we will use the additive notation in $\K2(k)$. The above presentation of $\K2(k)$ is independent of $n$ and a long root $\alpha$.

Note that $c$ may be regarded as a mapping $c\colon k^{\times}\times k^{\times} \to \K2(k)$. More generally, for an arbitrary abelian group $A$, every mapping $c\colon k^{\times}\times k^{\times} \to A$ satisfying (S1), (S2) and (S3) is called a \emph{symplectic Steinberg symbol}. For example, the mapping (\textasteriskcentered) is such.

An important consequence of the above theory is the existence of a one-to-one correspondence between symplectic Steinberg symbols on $k$ with values in $A$ and central extensions of $\Sp_{2n}(k)$ by $A$ up to equivalence (see for example \cite[Section 11]{milnor}). 

More precisely, suppose $A \hookrightarrow G' \overset{\pi'}{\twoheadrightarrow} \Sp_{2n}(k)$ is a central extension of $\Sp_{2n}(k)$. Then, since $\St(k)$ is the universal central extension of $\Sp_{2n}(k)$, there is a unique mapping $f$ such that the following diagram commutes
\[\xymatrix{
 {}\K2(k) \ar@{^{(}->}[rr] \ar[d]^-{f_{|\K2(k)}} & & \St(k) \ar@{>>}[rr]^-{\pi} \ar[d]^-{f} & & \Sp_{2n}(k) \ar[d]^-{\id} \\
 A        \ar@{^{(}->}[rr] & & G'       \ar@{>>}[rr]^-{\pi'} & & \Sp_{2n}(k).
}\]
Hence, the composition $c' = f\circ c$ is a symplectic Steinberg symbol. Conversely, by the Matsumoto-Moore theorem, every symplectic Steinberg symbol $c'\colon k^{\times}\times k^{\times} \to A$ is induced by a unique homomorphism $f_{c'}\colon\K2(k)\to A$ mapping $c(x,y)$ to $c'(x,y)$ for each $x,y \in k^{\times}$. For $N<\K2(k)\times A$ defined as $N = \left\{(x,-f_{c'}(x)) : x\in \K2(k)\right\}$ we obtain the following central extension of $\Sp_{2n}(k)$
\[\xymatrix{
A \ar@{^{(}->}[rr] & & \frac{\St(k)\times A}{N} \ar@{>>}[rr]^-{\pi'} & & \Sp_{2n}(k).
} \tag{\textasteriskcentered\textasteriskcentered\textasteriskcentered}\]


Notice also that if $H\colon \Sp_{2n}(k) \times \Sp_{2n}(k) \to \K2(k)$ is any 2-cocycle defining the extension $(**)$ (up to equivalence) and $c'\colon k^{\times}\times k^{\times} \to A$ is a symplectic Steinberg symbol, then the homomorphism $f_{c'}\colon\K2(k)\to A$ defined above induces a 2-cocycle $H_{c'}\colon\Sp_{2n}(k) \times \Sp_{2n}(k) \to A$ by putting $H_{c'}(a,a')=f_{c'}(H(a,a'))$. Then, the central extension of $\Sp_{2n}(k)$ by $A$ defined by means of $H_{c'}$ turns out to be equivalent to the extension (\textasteriskcentered\textasteriskcentered\textasteriskcentered). 

Let us collect some of the facts discussed above as a corollary so that we could easily refer to them later.

\begin{corollary}\label{cor:podsumowanie}
Let $c'\colon k^{\times}\times k^{\times} \to A$ be a symplectic Steinberg symbol. Then there exists a unique homomorphism $f_{c'}\colon\K2(k)\to A$ satisfying $f_{c'}(c(x,y))=c'(x,y)$ for all $x,y \in k^\times$. If $H\colon \Sp_{2n}(k) \times \Sp_{2n}(k) \to \K2(k)$ is a 2-cocycle defining the extension $(**)$, then the formula \[H_{c'}(a,a')=f_{c'}(H(a,a'))\] defines a 2-cocycle $H_{c'}\colon \Sp_{2n}(k) \times \Sp_{2n}(k) \to A$.
\end{corollary}

In order to apply Corollary \ref{cor:main}, we need to find non-split 2-cocycles on $\Sp_{2n}(k)$ with values in $\Z^n$ and with finite image. We will define such 2-cocycles using symplectic Steinberg symbols. 
We  use sections and 2-cocycles of \emph{finite width}.

\begin{definition} Let $s\colon \Sp_{2n}(k)\to \St(k)$ be a section of $\pi$ from (\textasteriskcentered\textasteriskcentered) such that $s(\id_{2n})=e$ and let $H\colon \Sp_{2n}(k) \times \Sp_{2n}(k) \to \K2(k)$ be the 2-cocycle induced by $s$. We say that:
\begin{enumerate}
\item \emph{$s$ has finite width} if there exists $N\in\N$ such that $s(M)$ can be written as a product of at most $N$ generators $u_\alpha(x)$ of $\St(k)$, for every $M\in\Sp_{2n}(k)$ (notice that by (A) $u_\alpha(x)^{-1} = u_\alpha(-x)$).
\item \emph{$H$ has finite width} if for some $N\in\N$, for every $M,M'\in\Sp_{2n}(k)$, $H(M,M')$ can be written as a sum of at most $N$ $\pm$ standard generators of $\K2(k)$.
\end{enumerate}
\end{definition}

\begin{remark}\label{rem:existence of s}
Sections of finite width of $\pi$ from (\textasteriskcentered\textasteriskcentered) satisfying $s(\id_{2n})=e$ always exist. 
\end{remark}
\begin{proof}
It is a well-known fact that there is $N\in\N$ such that every element of $\Sp_{2n}(k)$ is a product of at most $N$ generators $u_\alpha(x)$ (see the Bruhat decomposition in \cite[Section 8]{stein}). Using this fact, one can easily define a section of finite width. Namely, for any given element $M \in \Sp_{2n}(k)$,  write $M$ (not uniquely) as a product of at most $N$ elements of the form $u_\alpha(x)$, and define $s(M)$ as the same product but computed in $\St(k)$. 
\end{proof}

\begin{lemma} \label{lem:bound}
Let $s\colon \Sp_{2n}(k)\to \St(k)$ be a section of $\pi$ from (\textasteriskcentered\textasteriskcentered) of finite width and such that $s(\id_{2n})=e$. Then the associated 2-cocycle $H$ has also finite width.
\end{lemma}
\begin{proof}
Suppose the conclusion fails. Then for each $n\in\N$ we take $M_n,M'_n\in\Sp_{2n}(k)$ such that $H\left(M_n,M'_n\right)$ cannot be written as a sum of at most $N$ $\pm$ standard generators of $\K2(k)$. We take an ultraproduct of the sequence (\textasteriskcentered\textasteriskcentered), together with the section $s$ and the induced 2-cocycle $H$, over some non-principal ultrafilter $\U$ on $\N$. 
From now on, we will identify  $\Sp_{2n}(k^{\N}/\U)$ with $\Sp_{2n}(k)^{\N}/\U$ by the natural isomorphism.
Using 
the definition of a universal central extension (Definition \ref{def:uce}), we can write the following diagram:
\[\xymatrix{
 1 \ar@{^{(}->}[r] & \K2\left(k^{\N}/\U\right) \ar@{^{(}->}[rr] \ar[d]^{F_1=F_{|\K2}} & & \St\left(k^{\N}/\U\right) \ar@{>>}[rr]^{\pi_1} \ar[d]^F & & \Sp_{2n}\left(k^{\N}/\U\right) \ar@{>>}[r] \ar[d]^{\cong} \ar@/^1.7pc/[ll]_{s_1} & 1 \\
  1 \ar@{^{(}->}[r] & \K2(k)^{\N}/\U \ar@{^{(}->}[rr] & & \St(k)^{\N}/\U \ar@{>>}[rr]^{\pi^{\N}/\U} & & \Sp_{2n}\left(k\right)^{\N}/\U \ar@{>>}[r] \ar@/^1.7pc/[ll]_{s^{\N}/\U} & 1,
}\]
where 
\begin{itemize}
\item $F$ is given by Definition \ref{def:uce}, 
\item $F_1$ is the restriction of $F$ to $\K2\left(k^{\N}/\U\right)$,
\item $s_1$ is a section of $\pi_1$ satisfying $s_1(\id_{2n})=e$ and $F\circ s_1 = s^{\N}/\U$; the existence of $s_1$ is proved below.
\end{itemize}

By the uniqueness of $F$, we have that $F\bigg(u_{\alpha}\Big((x_n)_{n\in\N}/\U\Big)\bigg) = \Big(u_{\alpha}(x_n)\Big)_{n\in\N}/\U$ (because $F$ defined in a such a way makes the diagram commutative), and \[F_1\bigg(c\Big((x_n)_{n\in\N}/\U,(y_n)_{n\in\N}/\U\Big)\bigg) = \Big(c(x_n,y_n)\Big)_{n\in\N}/\U.\]


The existence of a section $s_1$ of $\pi_1$ satisfying $s_1(\id_{2n})=e$ and $F\circ s_1 = s^{\N}/\U$ follows from the fact that the image of $F$ contains the image of $s^{\N}/\U$ which, in turn, can be seen using  the above formula for $F$ and the fact $s$ has finite width and the root system $\Sigma=C_n$ is finite.

Let $H_1$ be the 2-cocycle induced by $s_1$. Clearly $F_1\circ H_1 = H^{\N}/\U$. Take $M, M'\in\Sp_{2n}(k^{\N}/\U)$ defined as $M=(M_n)_{n\in\N}/\U$ and $M'=(M'_n)_{n\in\N}/\U$. Then $H_1(M,M')$ is a finite sum of $\pm$ standard generators of $\K2\left(k^{\N}/\U\right)$.  Hence, also $H^{\N}/\U(M,M') = F_1(H_1(M,M'))$ can be expressed as a finite sum of images of $\pm$ standard generators. However, \[H^{\N}/\U(M,M')=\Big(H(M_n,M'_n)\Big)_{n\in\N}/\U,\] thus we get a contradiction with the choice of $M_n$ and $M'_n$.
\end{proof}


\begin{remark}\label{rem:computational}
Using similar methods to those in \cite{block}, it is possible to calculate explicit formulas in terms of Steinberg symbols for the 2-cocycle $H$ corresponding to the universal central extension (\textasteriskcentered\textasteriskcentered) for some particular section $s$ of finite width. In \cite{block}, this task has been done under the stronger assumption than (S1), that is under the assumption that $c$ is bilinear \cite[Page 146]{block}. For $\Sp_{2n}(k)=\SL_2(k)$ the following explicit formula for a 2-cocycle has been given by Matsumoto in \cite[5.12(a)]{matsumoto}: $H(M,M') = c\left(\chi(MM'),-\chi(M)^{-1}\chi(M')\right) - c(\chi(M),\chi(M'))$, where $\chi\colon \Sp_{2n}(k)\to k$ is defined as $\chi\left( \begin{array}{cc}
s & t \\
u & v \\ \end{array}\right) = v$ or $u$, depending on whether $u=0$ or not.
\end{remark}

To get Lemma \ref{lem:bound} by methods mentioned in the last remark would require very long computations, whereas our argument involving ultraproducts is short.
On the other hand, however, using this computational method, one can deduce that the formulas for the 2-cocycle in terms of finitely many standard generators of $\K2(k)$ depend in a `definable (in the field $k$) way' on the entries of the matrices from $\Sp_2(k)$, which allows to deduce definability of  2-cocycles $H_{c'}$ defined in Corollary \ref{cor:podsumowanie} in some concrete examples below.

Corollary \ref{cor:podsumowanie} together with Lemma \ref{lem:bound} give us the following conclusion.

\begin{corollary}\label{cor:finite image}
Assume that $c'\colon k^{\times}\times k^{\times} \to A$ is a symplectic Steinberg symbol and let $H\colon \Sp_{2n}(k) \times \Sp_{2n}(k) \to \K2(k)$ be a 2-cocycle as in Lemma \ref{lem:bound}. Let $H_{c'}\colon \Sp_{2n}(k) \times \Sp_{2n}(k) \to A$ be obtained from $H$ and $c'$ as it was described in Corollary \ref{cor:podsumowanie}. Then $\im(H_{c'})\subseteq (\im(c')-\im(c'))^N$ for some natural $N$. In particular, if $c'$ has finite image, then $H_{c'}$ has finite image. 
\end{corollary} 


Now, we make a few easy general observations, relating splitness with perfectness. We will use some of them to get non-splitness of 2-cocycles $H_{c'}$ considered in the last corollary. 

\begin{remark}\label{rem:nice and easy}
Suppose $\widetilde{G}$ is the central extension of a group $G$ by a non-trivial abelian group $A$ defined by means of a 2-cocycle $h \colon G \times G \to A$. Then we have:
\begin{enumerate}
\item If $\left[\widetilde{G},\widetilde{G}\right] \cap A$ is non-trivial, then $h$ is non-split.
\item If $\widetilde{G}$ is perfect, then $h$ is non-split.
\item If $G$ is perfect and $h$ non-split, then $\left[\widetilde{G},\widetilde{G}\right] \cap A$ is non-trivial.
\item If $G$ is perfect and $h$ is strongly non-split (i.e., for any proper $A'<A$ the induced 2-cocycle $\overline{h}: G \to A/A'$ is non-split), then $\widetilde{G}$ is perfect.
\item Assume $G$ is perfect and it acts trivially on an abelian group $C\geq A$. Then $h\colon G \times G \to A$ is non-split iff $h \colon G \times G \to C$ (i.e., $h$ treated as a 2-cocycle with values in $C$) is non-split. 
\end{enumerate}
\end{remark}

\begin{proof}
(1) If $h$ was split, we would get an isomorphism $\Phi \colon  \widetilde{G} \to A \times G$ commuting with the projections on $G$. Then $\{ 0 \} \ne \Phi\left[\left[\widetilde{G}, \widetilde{G}\right] \cap A\right]=\left[A \times G,A\times G\right] \cap A=\{ 0\}$, a contradiction.\\
(2) follows from (1).\\
(3) Suppose for a contradiction that $\left[\widetilde{G},\widetilde{G}\right] \cap A$ is trivial. Let 
\[\xymatrix{
 1 \ar@{^{(}->}[r] & C \ar@{^{(}->}[r] & U \ar@{>>}[r]^-{\pi'} & G \ar@{>>}[r] & 1.
} \]
be the universal central extension of $G$ (it exists since $G$ is perfect). Then there exists a homomorphism $\Phi \colon U \to \widetilde{G}$ which commutes with the projections on $G$. Since the universal central extension is always perfect, $\Phi[C]=\Phi\left[\left[U,U\right] \cap C\right] \subseteq \left[\widetilde{G},\widetilde{G}\right] \cap A = \{ 0 \}$. Thus, $\Phi$ is trivial on $C$, and so the section $s \colon G \to \widetilde{G}$ of the projection of $\widetilde{G}$ onto $G$ defined by $s(g)= \Phi\left[\pi'^{-1}(g)\right]$ is a homomorphism. This implies that $h$ is split, a contradiction.\\
(4) follows from (3). Namely, suppose for a contradiction that $A':=\left[\widetilde{G}, \widetilde{G}\right] \cap A$ is a proper subgroup of $A$. Then, applying (3) to the central extension of $G$ by $A/A'$ defined by the 2-cocycle $\overline{h}\colon G \times G \to A/A'$ induced by $h$, we get that $\overline{h}$ is split, a contradiction.\\
(5) follows from (1) and (3).
\end{proof}

\begin{corollary}\label{cor:better version}
Suppose that $c'\colon k^{\times}\times k^{\times} \to A$ is a symplectic Steinberg symbol and $\chr(k)=0$. By $H_{c'}$ we denote the 2-cocycle considered in Corollary \ref{cor:finite image}. In (2) and (3) below, we assume that $H$ is induced by a section $s$ of finite width satisfying $s(\id_{2n})=e$ and additionally such that the image $s[\Sp_{2n}(\Q)]$ is contained in the subgroup of $\St(k)$ generated by $\{ u_\alpha(x): \alpha  \in \Sigma, x \in \Q\}$ (note that the existence of at least one such $s$ follows as in Remark \ref{rem:existence of s}).
\begin{enumerate}
\item If $c'$ is non-trivial, then $H_{c'}$ is non-split.
\item Consider any $G$ with $\Sp_{2n}(\Q)\leq G \leq \Sp_{2n}(k)$. If $c'_{|\Q^\times \times \Q^{\times}}$ is non-trivial, then ${H_{c'}}_{| G \times G}$ is non-split.
\item Consider any $G$ with $\Sp_{2n}(\Q)\leq G \leq \Sp_{2n}(k)$. Assume that $\Sp_{2n}(\Q)$, $G$, $A$ and the 2-cocycle ${H_{c'}}_{|G\times G}$ are $B$-definable in some first order structure $\G$. Let $\G^*\succ \G$ be a monster model. Then, if $c'_{|\Q^\times \times \Q^{\times}}$ is non-trivial, the 2-cocycle ${{H_{c'}}^*}_{| {G^*}^{00}_B\times{G^*}^{00}_B} \colon {G^*}^{00}_B\times{G^*}^{00}_B \to A^*$ is non-split.
\end{enumerate}
\end{corollary}

\begin{proof}
(1) As a set, $\St(k)$ can be presented as the product $\K2(k) \times \Sp_{2n}(k)$. Let $\widetilde{\Sp_{2n}(k)_{c'}}$ be the central extension of $\Sp_{2n}(k)$ by $A$ defined by means of the 2-cocycle $H_{c'}$. Then the function $\Phi \colon \St(k) \to \widetilde{\Sp_{2n}(k)_{c'}}$ defined as $f_{c'} \times \id$ is a homomorphism commuting with projections
 onto $\Sp_{2n}(k)$ (where $f_{c'}$ is provided by Corollary \ref{cor:podsumowanie}). Since $\left[\St(k),\St(k)\right] \cap \K2(k)=\K2(k)$ and $f_{c'} \colon \K2(k) \to A$ is not identically $0$, we get 
\begin{eqnarray*}
\left[\widetilde{\Sp_{2n}(k)_{c'}},\widetilde{\Sp_{2n}(k)_{c'}}\right] \cap A \supseteq \Phi\left[\left[\St(k),\St(k)\right] \cap \K2(k)\right] &=& \Phi\left[\K2(k)\right] \\
 &=& f_{c'}\left[\K2(k)\right] \ne \{0\}.
\end{eqnarray*}
 By Remark \ref{rem:nice and easy}(1), the conclusion is clear.\\
%
%
%
%
(2) Let $P$ be the subgroup of $\K2(k)$ generated by $\left\{ c(x,y) : x,y \in \Q^\times\right\}$. By our assumption on $H$, there is a 2-cocycle defining $\St(\Q)$ induced by a section of finite width and defined by the same formulas in terms of $c(x,y)$, $x,y \in \Q^\times$, as the 2-cocycle $H_{| \Sp_{2n}(\Q) \times \Sp_{2n}(\Q)}$. Therefore, by the perfectness of $\St(\Q)$, we get that, working inside $\St(k)$, one has 
\[\left[P \times \Sp_{2n}(\Q), P \times \Sp_{2n}(\Q)\right] \cap P=P.\]
Thus, arguing as in (1), we get the conclusion of (2).\\
(3) Since $\Sp_{2n}(\Q)$ is absolutely connected, we have \[\Sp_{2n}(\Q) \leq \Sp_{2n}(\Q)^* ={\Sp_{2n}(\Q)^*}^{00} \leq {G^*}^{00}_B.\] So, the conclusion follows from (2). 
\end{proof}

We give a criterion for $H_{c'}$ to be strongly non-split in the sense of Corollary \ref{cor:main}$(4)$ and Proposition \ref{prop:ext}.



\begin{corollary}\label{cor:strong non-splitting}
Suppose that $c'\colon k^{\times}\times k^{\times} \to A$ is a symplectic Steinberg symbol and $\chr(k)=0$. 
Assume that $A=\Z$ and $c'(-1,-1)=1$. 

Take $H_{c'}$ as in Corollary \ref{cor:better version}. If $G$ is any subgroup of $\Sp_{2n}(k)$ containing $\Sp_{2n}(\Q)$, then $H_{c'}$ restricted to $G$ is strongly non-split, that is for every proper subgroup $A'\lneqq A$ the induced 2-cocycle $\overline{H_{c'}}\colon G\times G\to A/A'$ is non-split.
\end{corollary}
\begin{proof}
Notice that $\overline{H_{c'}}=H_{\overline{c'}}$, where $\overline{c'}\colon k^{\times}\times k^{\times} \to A/A'$ is the induced Steinberg symbol. Since $\overline{c'}(-1,-1) \ne 0 +A'$, the conclusion follows from Corollary \ref{cor:better version}(2). 
\end{proof}

Now, we are ready to give new classes of examples of groups in which the smallest invariant subgroup of bounded index is a proper subgroup of the smallest type-definable subgroup of bounded index.

\begin{example}\label{ex:main}
Suppose $k$ is an ordered field with the order denoted by $<$ (e.g. $k=\Q$ or $k$ is an arbitrary subfield of $\R$ with the natural order). 
One can check that the following mapping $c' \colon k^{\times} \times k^{\times} \to {\Z}$ is a symplectic Steinberg symbol \[c'(x,y) = \left\{
  \begin{array}{r l}
    1 & \quad \text{if } x < 0 \text{ and } y < 0 \\
    0 & \quad \text{otherwise}\\
  \end{array} \right.. \tag{\textasteriskcentered\textasteriskcentered\textasteriskcentered\textasteriskcentered}\]

%
%
Let $H \colon \Sp_{2n}(k) \times \Sp_{2n}(k) \to \K2(k)$ be a 2-cocycle defining the universal central extension of $\Sp_{2n}(k)$ as in Lemma \ref{lem:bound}. By \[H_{c'}\colon \Sp_{2n}(k) \times \Sp_{2n}(k) \to \Z\] we denote the 2-cocycle obtained from $H$ and $c'$ as it was described in Corollary \ref{cor:podsumowanie}.

Let $\G = ((\Z,+), (k,+,\cdot,<), H_{c'})$, $G=\Sp_{2n}(k)$ and $A=(\Z,+)$. The groups $G$ and $A$ are $\emptyset$-definable in $\G$. Assume the action of $G$ on $A$ is trivial, and let $\widetilde{G}$ (a central extension of $G$ by $\Z$) be defined by means of the 2-cocycle $H_{c'}$. Let $\G^*=((\Z^*,+), (k^*,+,\cdot,<))\succ \G$ be a monster model. Put as $A^*_1$ the connected component ${\Z^*}^0$ of the pure group $(\Z^*,+)$ (i.e., the intersection of all groups $n\Z^*$, $n \in \N \setminus \{ 0 \}$)
%
and $B=\{1\}$ (where $1$ is from the sort $(\Z,+)$).

Then the assumptions of Corollary \ref{cor:main} are satisfied. So, we get $\widetilde{G^*}^{000}_B \ne \widetilde{G^*}^{00}_B=\widetilde{G^*}$, where $\widetilde{G^*}$ is the interpretation of $\widetilde{G}$ in the monster model $\G^*$. Moreover, the quotient $\widetilde{G^*}^{00}_B / \widetilde{G^*}^{000}_B$ is abelian. In fact, $\widetilde{G^*}^{000}_B=\left({\Z^*}^0 +\Z\right) \times G^*$, and $\widetilde{G^*}^{00}_B/\widetilde{G^*}^{000}_B$ is isomorphic to $\widehat{\Z}/\Z$, where $\widehat{\Z}$ is the profinite completion of $\Z$.


\end{example}

\begin{proof}
The fact that the image of $H_{c'}$ is finite follows from Corollary \ref{cor:finite image}.
The assumptions (2) and (5) of Corollary \ref{cor:main} are clearly satisfied (see the proof of Example \ref{ex:PC1}). The assumption (3) (even ${G^*}^{000}_B=G^*$) follows from the absolute connectedness of $\Sp_{2n}(k)$. Finally, (1) and (4) are true by Corollary \ref{cor:strong non-splitting}.

The fact that $\widetilde{G^*}^{00}_B / \widetilde{G^*}^{000}_B$ is abelian follows from the absolute connectedness of $\Sp_{2n}(k)$ and Proposition \ref{prop:quotient}. 


To get the desired description of this quotient, one should apply Propositions \ref{prop:quotient} and \ref{prop:density} together with Claim 1 in the proof of Theorem \ref{thm:main}. Namely, by this claim applied to $H:={\Z^*}^0$, we have that $\widetilde{G^*}^{000}_B \cap \Z^* = {\Z^*}^0 +n\Z$ for some $n \in \N$. Using this together with the fact that $\widetilde{G^*}^{00}_B=\widetilde{G^*}$ and Proposition \ref{prop:quotient}, we get \[\widetilde{G^*}^{00}_B/\widetilde{G^*}^{000}_B \cong \left( \Z^*/{\Z^*}^0\right) /\left( ({\Z^*}^0+n\Z)/{\Z^*}^0\right).\]
But Proposition \ref{prop:density} tells us that $({\Z^*}^0+n\Z)/{\Z^*}^0$ is a dense subgroup of $\Z^*/{\Z^*}^0$, which implies that $n=1$. Thus, once again using Claim 1 in the proof of Theorem \ref{thm:main}, we conclude that
\[\widetilde{G^*}^{000}_B = ({\Z^*}^0 + \Z) \times G^*\; \; \mbox{and}\;\; \widetilde{G^*}^{00}_B/\widetilde{G^*}^{000}_B \cong \widehat{\Z}/\Z.\]




\end{proof}


\begin{remark}
Starting from the 2-cocycle $H$ induced by some particular section existence  of which is mentioned in Remark \ref{rem:computational} (see also the paragraph following \ref{rem:computational}), one could prove that $H_{c'}$ from Example \ref{ex:main} is $B$-definable in  $((\Z,+), (k,+,\cdot,<))$.
\end{remark}


In the same way as Example \ref{ex:PC2} was obtained from Example \ref{ex:PC1}, for an arbitrary ordered field $(k,+,\cdot,<)$ the above example yields an extension $\widetilde{G}$ of $\Sp_{2n}(k)$ by $\SO_2(k)$ which is definable in $\G:=(k,+,\cdot,<)$ and such that $\widetilde{G^*}^{00}_B \ne \widetilde{G^*}^{000}_B$, where 
$B:=\{g\}$ for some $g \in \SO_2(k)$ of infinite order. 



Next, we generalize the situation from Example \ref{ex:main} in the following way.



\begin{example}\label{ex:main_gen}
Suppose $k$ is  an ordered field, and $c' \colon k^\times \times k^\times \to \Z$ the Steinberg symbol defined in Example \ref{ex:main}. Let $H_{c'}$ be a 2-cocycle as in Corollary \ref{cor:better version} (i.e., with the additional assumption on $H$).
Let $G$ be an arbitrary group with $\Sp_{2n}(\Q)\leq G \leq \Sp_{2n}(k)$, and let $\widetilde{G}$ be the central extension of $G$ by $\Z$ corresponding to $H_{c'}$. Let $A=\Z$ and $B=\{1\}$. 
Assume that $\Sp_{2n}(\Q)$, $G$ and $A$ are $\emptyset$-definable and the 2-cocycle ${H_{c'}}_{|G\times G}$ is $B$-definable in a first order structure $\G$. For example, $\G$ might be the two-sorted structure with the disjoint sorts $(\Z,+)$ and $(k,+,\cdot,<)$ together with predicates for $G$, $\Sp_{2n}(\Q)$ and ${H_{c'}}_{|G\times G}$. Let $\G^*\succ \G$ be a monster model. Put $A^*_1={\Z^*}^0$, the intersection of all groups $n\Z^*$, $n \in \N \setminus \{ 0 \}$. Then the assumptions of the first part of Theorem \ref{thm:main} (in fact, even the strong version of Assumption (i)) are satisfied, so $\widetilde{G^*}^{000}_B \ne \widetilde{G^*}^{00}_B$.


If, moreover, ${G^*}^{000}_B =G^*$, then the assumptions of Corollary \ref{cor:main} are satisfied too, so $\widetilde{G^*}^{00}_B=\widetilde{G^*}$, and $\widetilde{G^*}^{00}_B / \widetilde{G^*}^{000}_B$ is abelian.
\end{example}

\begin{proof}
Assumption (ii) of Theorem \ref{thm:main} is clearly satisfied, whereas the strong version of Assumption (i) follows from Corollary \ref{cor:better version}(3).
%
%
For the `moreover' part, it is enough to use Corollary \ref{cor:strong non-splitting} and Proposition \ref{prop:quotient}.
\end{proof}

Starting from Examples \ref{ex:main} or \ref{ex:main_gen}, the results of Section \ref{sec:finite} yield more general classes of examples, some of which are briefly discussed below. In order to avoid a clash of notation (in Section \ref{sec:finite}, $H$ was a group, whereas in this section, $H$ is a 2-cocycle), the 2-cocycles $H$ and $H_{c'}$ considered in Example \ref{ex:main} will be denoted by $h$ and $h_{c'}$, respectively.

\begin{example}\label{ex:extension of main}
Consider the situation from Example \ref{ex:main} (as it is mentioned above, instead of $H$ and $H_{c'}$ we write $h$ and $h_{c'}$). Let 
\[ \xymatrix{
 1 \ar@{^{(}->}[r] & \ker (f) \ar@{^{(}->}[r] & H \ar@{>>}[r]^-{f} & G \ar@{>>}[r] & 1}\]
be an extension of $G$ by $\ker(f)$. Let $\H$ be any expansion of $\G$ in which $H$ and $f$ are $\emptyset$-definable (e.g. $\H$ is the expansion of $\G$ by the new sort $H$ together with the function $f$), and let $\H^*\succ \H$ be a monster model. Assume additionally that $\Hom (\ker ({f^*}_{|{H^*}^{00}_B}), \Z)$ is trivial (where $f^*$ in the interpretation of $f$ in $\H^*$). Put $h': =h_{c'} \circ (f,f) \colon H \times H \to \Z$ a 2-cocycle definable in $\H$ over $B$. 
Let $\widetilde{H}$ be the extension of $H$ by $\Z$ corresponding to $h'$. Then ${{h'}^*}_{| {H^*}^{00}_B \times {H^*}^{00}_B} \colon {H^*}^{00}_B \times {H^*}^{00}_B \to \Z$ is non-split, and $\widetilde{H^*}^{000}_B \ne \widetilde{H^*}^{00}_B$.

To see a concrete example arising in this way, take as $H$ the extension of $\Sp_{2n}(k)$ (where $k$ is an ordered field) by the divisible hull ${\Q} \otimes_\Z \K2(k)$ of $\K2(k)$ corresponding to the 2-cocycle $h$, and as $f\colon H \to \Sp_{2n}(k)$ the projection on the second coordinate. More generally, start from any extension of $\Sp_{2n}(k)$ by an abelian group $C$ and take as $H$ the group obtained from this extension by replacing $C$ by its divisible hull. 

Another family of examples arising in this way is formed by the groups of the form $\widetilde{H}$ for $H$ ranging over all finite extension of $\Sp_{2n}(k)$.
\end{example}

\begin{proof}
Using Corollary \ref{cor:ext}, the conclusion follows from the observations that $\Sp_{2n}(k)$ is absolutely connected and that the assumptions of Theorem \ref{thm:main} (including the strong version of Assumption (i)) are satisfied in Example \ref{ex:main}.
\end{proof}

\begin{example}\label{ex:produkty_przyklad}
Consider the situation from Example \ref{ex:main_gen}. Let $H$ be the product of groups $K \times G$ for an arbitrary group $K$. We define $\H$ as the expansion of $\G$ obtained by adding a new sort, consisting of the pure group structure on H and the projection $f \colon H \to \Sp_{2n}(k)$ on the second coordinate. Let $\H^* \succ \H$ be a monster model and $f^*$ the interpretation of $f$ in it. Put $h': =h_{c'} \circ (f,f) \colon H \times H \to \Z$ a 2-cocycle definable in $\H$ over $B$. 
%
%
Let $\widetilde{H}$ be the extension of $H$ by $\Z$ corresponding to $h'$. Then ${{h'}^*}_{| {H^*}^{00}_B \times {H^*}^{00}_B} \colon {H^*}^{00}_B \times {H^*}^{00}_B \to \Z$ is non-split, and $\widetilde{H^*}^{000}_B \ne \widetilde{H^*}^{00}_B$. 
\end{example}

\begin{proof}
Using Remark \ref{rem:produkty}, the conclusion follows from the fact that the assumptions of the first part of Theorem \ref{thm:main} are satisfied in Example \ref{ex:main_gen}.
\end{proof}

Although in Example \ref{ex:main} the assumptions of both parts of Theorem \ref{thm:main} (even of Corollary \ref{cor:main}) are satisfied (and, in consequence, $\widetilde{G^*}^{000}_B \ne \widetilde{G^*}^{00}_B =\widetilde{G^*}$), in Example \ref{ex:extension of main}, we only concluded that the assumptions of the first part of Theorem \ref{thm:main} hold, and, in consequence, $\widetilde{H^*}^{000}_B \ne \widetilde{H^*}^{00}_B$. 
It is rather clear that in general, we can not expect that $\widetilde{H^*}^{00}_B=\widetilde{H^*}$. For example, take $H:=\Z/n\Z \times \Sp_{2n}(k)$, $f \colon H \to \Sp_{2n}(k)$ the projection on the second coordinate, and 
$\H=\G=((\Z,+),(k,+,\cdot,<),h_{c'})$ where $H$ is definable in the obvious way. Then $\{ 0 +n\Z \} \times \Sp_{2n}(k)$ is a finite index, $B$-definable (in $\H$) subgroup of $H^*$, so ${H^*}^{00}_B \ne H^*$. Therefore, $\widetilde{H^*}^{00}_B \ne \widetilde{H^*}$. 

However, from
%
%
Proposition \ref{prop:ext}, we know that if $\ker (f)$ is a finite abelian group and some 2-cocycle $h_f$ corresponding to the extension
\[ \xymatrix{
 1 \ar@{^{(}->}[r] & \ker (f) \ar@{^{(}->}[r] & H \ar@{>>}[r]^-{f} & \Sp_{2n}(k) \ar@{>>}[r] & 1}\]
is strongly non-split, then $H$ is absolutely connected (since $\Sp_{2n}(k)$ is such). The next example shows that even in such a situation, it may happen that $\widetilde{H^*}^{00}_B \ne \widetilde{H^*}$.

\begin{example}\label{ex:counter-example}
Consider the situation from Example \ref{ex:main}. In particular, $G=\Sp_{2n}(k)$, $A=\Z$ and the extension $\widetilde{G}$ is given by the 2-cocycle $h_{c'} \colon G \times G \to A$. Let $h_f \colon G \times G \to \Z/n\Z$ (for some $n \in {\N} \setminus \{ 0 \}$) be the 2-cocycle induced by $h_{c'}$, $H$ be the corresponding extension of $G$ by $\Z/n\Z$, and $f\colon H \to G$ be the projection on the second coordinate. Recall that the 2-cocycle $h' \colon H \times H \to \Z$ is defined as $h_{c'} \circ (f,f)$, and $\widetilde{H}$ is the corresponding extension of $H$ by $\Z$. 
Define $\H = \G = ((\Z,+), (k,+,\cdot,<),h_{c'})$. Then everything is $B$-interpretable in $\G$. As usual, $\H^*=\G^* \succ \G$ is a monster model.

Then $H$ is absolutely connected (so ${H^*}^{00}_B=H^*$), but $\widetilde{H^*}^{000}_B \ne \widetilde{H^*}^{00}_B \ne \widetilde{H^*}$.


\end{example}

\begin{proof}
Since, by Example \ref{ex:main}, $h_{c'}$ is strongly non-split, so is $h_f$. Thus, the absolute connectedness of $H$ follows from Proposition \ref{prop:ext} and the fact that $\Sp_{2n}(k)$ is absolutely connected.

The fact that $\widetilde{H^*}^{000}_B \ne \widetilde{H^*}^{00}_B$ was proved in Example \ref{ex:extension of main}. It remains to show that $\widetilde{H^*}^{00}_B \ne \widetilde{H^*}$.

Of course, $\widetilde{H}/n\Z$ is B-definably isomorphic with the extension of $H$ by $\Z / n\Z$ corresponding to the 2-cocycle $\overline{h'} \colon H \times H \to \Z/n\Z$ induced by $h'$; denote this extension by $K$. 


Let $i_n \colon \widetilde{G} \to H$ be the $B$-definable homomorphism defined by $i_n(a,x)=(a +n\Z, x)$. Then $\overline{h'}=i_n \circ h_{c'}\circ (f,f)$ (after the appropriate identifications). 

Define a section $z \colon G \to \widetilde{G}$ by $z(x)=(0,x)$. Then $z_n := i_n \circ z \colon G \to H$ is a section of $f$.

We easily get 
\[\overline{h'}(x,y)=F(x)+F(y)-F(xy),\]
%
%
where $F \colon H \to \Z/n\Z$ is defined by $F(x)=z_n(f(x))x^{-1}$. This means that $\overline{h'}$ is split via the $B$-definable function $F$. Therefore, $K$ is $B$-definably isomorphic with the product $\Z / n\Z \times H$. Hence, $K^*$ is $B$-definably isomorphic with $\Z^* / n\Z^* \times H^*$. Since clearly $\left(\Z^* / n\Z^* \times H^*\right)^{00}_B \leq \{ 0 +n\Z^*\} \times H^* \ne \Z^* / n\Z^* \times H^*$, we get that ${K^*}^{00}_B \ne K^*$, which implies $\left(\widetilde{H^*}/n\Z^*\right)^{00}_B \ne \widetilde{H^*}/n\Z^*$, and so $\widetilde{H^*}^{00}_B \ne \widetilde{H^*}$.
\end{proof}


We finish with a discussion on Steinberg symbols.
Note that a given field may have many different orders. Each of them gives rise to some symplectic Steinberg symbol, yielding various classes of examples of groups described above.

Recall that a field $k$ is an ordered field (with respect to some order) if and only if it is formally real (i.e. $-1$ is not a sum of squares in $k$). Corollaries \ref{cor:finite image} and \ref{cor:better version} and Example \ref{ex:main} lead to the following question:
\begin{quote}
Can a non-formally real field $k$ have a non-trivial symplectic Steinberg symbol $c\colon k^{\times}\times k^{\times} \to \Z^n$ with finite image? 
\end{quote}
We answer this question in the negative for fields of characteristic different from $2$.

For a field $k$, by $S(k)$ we denote the set of sums of squares $\left\{\sum_{i=1}^{n} a_i^2 : a_i\in k^{\times}\right\}$ of $k$.

\begin{proposition}
Suppose $k$ is a field, $\chr(k)\ne 2$ and $c\colon k^{\times}\times k^{\times} \to A$ is a symplectic Steinberg symbol with finite image, where $A$ is a torsion free abelian group. Then for every $0\ne s\in S(k)$ and $t\in k^{\times}$ \[c(s,t)=c(t,s)=0.\]
\end{proposition}

\begin{proof}
We use the relations (S1), (S2) and (S3) as well as the following formulas, which are consequences of (S1) -- (S3) (see \cite[Proposition 5.7, p. 28]{matsumoto}): for $x,y\in k^{\times}$
\begin{itemize}
\item[(1)] $c(x,y) = c(y^{-1},x)$,
\item[(2)] $c(x,y) = c(x,-xy)$,
\item[(3)] the mapping $t \mapsto c(x,t^2)$ is a homomorphisms from $k^{\times}$ to $A$.
\end{itemize}

 We will prove that $c(s,t)=c(t,s)=0$ whenever $s=\sum_{i=1}^{n} a_i^2$ for $a_i\in k^{\times}$ and  $t\in k^{\times}$. We will do it by induction on $n$.

Case $n=1$. Since the image of $c$ is finite and $A$ is torsion free, (3) implies that the mapping $t \mapsto c(x,t^2)$ is trivial for an arbitrary $x\in k^{\times}$. Thus, by (1), $c(s,t)=c(t,s)=0$.

Before proving the inductive step, we prove the following relations: for $x,y,z\in k^{\times}$
\begin{itemize}
\item[(4)] $c(xy^2,z)=c(x,zy^2)=c(x,z)$,
\item[(5)] $c(x,y)=c(y,x)$,
\item[(6)] $c(x,-1)=c(x,y) + c(x,-y)$.
\end{itemize}

(4) Using (S1) and the case $n=1$, we have $c(x,zy^2)=c(x,z)+c(xz,y^2)-c(z,y^2)= c(x,z)$. The proof of $c(xy^2,z)=c(x,z)$ is similar.

(5) By (1) and (4), $c(x,y)=c(y^{-1},x)=c(y^{-1}y^2,x)=c(y,x)$.

(6) By (S1), we have $c(x,-1)=c(x,y(-y^{-1})) = c(x,y) + c(xy,-y^{-1}) - c(y,-y^{-1})$. Moreover, by (2), $c(y,-y^{-1})=c(y,1)=0$ and by (5), (2) and (4), $c(xy,-y^{-1}) = c(-y^{-1},xy) = c(-y^{-1},x)=c(-y,x) = c(x,-y)$. Hence we get (6).

Inductive step $n\rightarrow n+1$. Let $s=\sum_{i=1}^{n+1} a_i^2$. First, we prove that $c(s,-t) = c(s,t)$. 

We may assume that $s \ne a_1^2$, because otherwise we are done by the case $n=1$. Then $1 + \sum_{i=2}^{n+1} \left(\frac{a_i}{a_1}\right)^2 \ne 1$, hence we have
\begin{eqnarray*}
c(s,t) &\overset{(4)}{=}& c\left( 1 + \sum_{i=2}^{n+1} \left(\frac{a_i}{a_1}\right)^2,t \right) \overset{(S3)}{=} c\left(1 + \sum_{i=2}^{n+1} \left(\frac{a_i}{a_1}\right)^2, -\left(\sum_{i=2}^{n+1} \left(\frac{a_i}{a_1}\right)^2\right)t \right) \\
       &\overset{(4)}{=}& c\left(s, \left(\sum_{i=2}^{n+1} \left(\frac{a_i}{a_1}\right)^2\right)(-t) \right).
\end{eqnarray*}
By (S1) and the induction assumption,
\begin{eqnarray*}
c\left(s, (-t)\left(\sum_{i=2}^{n+1} \left(\frac{a_i}{a_1}\right)^2\right) \right) &=& c(s,-t) + c\left(-st,\sum_{i=2}^{n+1} \left(\frac{a_i}{a_1}\right)^2 \right) - c\left(-t,\sum_{i=2}^{n+1} \left(\frac{a_i}{a_1}\right)^2\right) \\
&=& c(s,-t). 
\end{eqnarray*}
If $t=-1$, then $c(s,-1)=c(s,1)=0$ by the case $n=1$. Hence, by (6), $0=c(s,-1) = c(s,t)+c(s,-t) = 2c(s,t)$, so $c(s,t)=0$ and $c(t,s)=0$ by (5).
\end{proof}

\begin{corollary}
If $k$ is a non-formally real field and $\chr(k)\ne 2$, then every symplectic Steinberg symbol $c\colon k^{\times}\times k^{\times} \to A$ with finite image, where $A$ is a torsion free abelian group, is trivial.
\end{corollary}
\begin{proof}
By assumption, $-1=a_1^2 + \ldots + a_n^2$ is a sum of squares in $k$. Then every $x\in k^{\times}$ can be written as a sum of squares $x=\left(\frac{x+1}{2}\right)^2 - \left(\frac{x-1}{2}\right)^2 = \left(\frac{x+1}{2}\right)^2 + \left(a_1\frac{x-1}{2}\right)^2 + \ldots + \left(a_n\frac{x-1}{2}\right)^2$. Hence, by the previous proposition, $c$ is trivial.
\end{proof}

\begin{question}
Does there exist an infinite field of characteristic 2 possessing a non-trivial symplectic Steinberg symbol with finite image contained in a torsion free abelian group?
\end{question}

There exists a more general theory of central extensions of Chevalley groups (see \cite{matsumoto}) via symbols. However, when $G$ is not of symplectic type, then every symbol $c'\colon k^{\times}\times k^{\times} \to A$ is bimultiplicative, so $c'(xy,z)=c'(x,z)+c'(y,z)$. Therefore, if $A=\Z^n$, then every non-trivial $c'$ has infinite image, so our approach cannot be applied.

The following question is interesting.

\begin{question}
Does there exist an abelian group $G$ (defined over $\emptyset$ in a monster model) for which $G^{000}_\emptyset \ne G^{00}_\emptyset$. Can one find such a group using Theorem \ref{thm:main}?
\end{question}


\subsection{Applications of Corollary \ref{cor:main2}}\label{subsection:4.2}

In the first part of this subsection, we will be using various quasi-characters considered in \cite[Section 5]{Gr} in order to produce the desired 2-cocycles with finite image. In the final part, we give a detailed analysis of the extension of $\SL_2(\Z)$ by $\Z$ defined by the restriction to $\SL_2(\Z)\times \SL_2(\Z)$ of the 2-cocycle defining $\widetilde{\SL_2(\R)}$.

We recall the notions of quasi-character and pseudo-character, restricting ourselves to the integer-valued functions. In the next two definitions, $G$ is an arbitrary group.
\begin{definition}
A function $f \colon G \to \Z$ is said to be a \emph{quasi-character} if there exists a constant $C \geq 0$ such that for all $x,y \in G$ one has 
\[|f(x)+f(y)-f(xy)| \leq C.\]
Equivalently, the image of $f(x)+f(y)-f(xy)$ is finite.
\end{definition}

\begin{definition}
A function $f \colon G \to \Z$ is said to be a \emph{pseudo-character} if it is a quasi-character such that for all $x \in G$ and $n \in \Z$ one has \[f(x^n)=nf(x).\]
\end{definition}
The reference for the next there parts are the corresponding parts of \cite[Section 5]{Gr}.
%


\subsubsection{Extensions of free groups}

%
$\F_m$ denotes the free group with $m$ free generators $a_1,\dots,a_m$. Assume that $m \geq 2$. Put $F=\{a_1,\dots,a_m\}$.

Let $W$ be a reduced word  over the alphabet $F \cup F^{-1}$. Define $f_W \colon \F_m \to \Z$ by
\[f_W(g)= \mbox{number of occurrences of $W$ in $g$}\; - \; \mbox{number of occurrences of $W^{-1}$ in $g$},\]
where by an `occurrence of $W$ in $g$' we mean any occurrence of $W$ as a subword of the reduced word representing $g$. 

\begin{fact}
$f_W$ is a quasi-character on $\F_m$. Moreover, $f_W(g^{-1})=-f_W(g)$ for all $g \in \F_m$.
\end{fact}

We say that $W$ is without \emph{self-overlapping} if $B(W) \cap E(W) = \emptyset$, where $B(W)$ is the set of beginnings and $E(W)$ is the set of endings of $W$ (i.e., if $W=a_{i_1}^{\xi_1} \dots a_{i_k}^{\xi_k}$, then $B(W)=\left\{ a_{i_1}^{\xi_1},a_{i_1}^{\xi_1}a_{i_2}^{\xi_2},\dots, a_{i_1}^{\xi_1}\dots a_{i_{k-1}}^{\xi_{k-1}}\right\}$ and $E(W)=\left\{a_{i_2}^{\xi_{i_2}} \dots a_{i_{k}}^{\xi_k},\dots, a_{i_{k-1}}^{\xi_{k-1}}a_{i_k}^{\xi_k},a_{i_k}^{\xi_k}\right\}$.)
For an element $g \in \F_m$, denote by $\overline{g}$ the reduced cyclic word corresponding to $g$.

Define $e_W \colon \F_m \to \Z$ by
\[e_W(g)= \mbox{number of occurrences of $W$ in $\overline{g}$}\; - \; \mbox{number of occurrences of $W^{-1}$ in $\overline{g}$}.\]

\begin{fact}
$e_W$ is a quasi-character on $\F_m$ satisfying $e_W(g^{-1})=-e_W(g)$. Moreover, if $W$ is without self-overlapping, then $e_W$ is a pseudo-character.
\end{fact}

Using the appropriate quasi-character's, one can find many examples to which Corollary \ref{cor:main2} can be applied. We give one such example.

\begin{example}\label{ex:new1}
Let $G=\F_m=\F_{a_1,\dots,a_m}$ (where $m \geq 2$) and  $A=(\Z,+)$. Take $f=f_{a_1^2}$ or $f=e_{a_1^2}$ (i.e., we take $W:=a_1^2$ in the discussion above). Define $h: \F_m \times \F_m \to \Z$ by $h(x,y)=f(x)+f(y)-f(xy)$ for $x,y \in \F_m$. 
Then $f$ is a quasi-character, so $h$ has finite image; in fact, $\im(h) = \{-1,0,1\}$. Take $B=\{ 1\}$. 
Let $\G$ be any expansion of $((\Z,+),(\F_m,\cdot),f)$, and let $\G^* \succ \G$ be a monster model. Let $\widetilde{G}$ be the central extension of $G$ by $\Z$ defined by means of $h$. Put $A^*_1=\bigcap_{n \in \N \setminus \{ 0 \}} n\Z^*$. Then the assumptions of Corollary \ref{cor:main2} are satisfied, so $\widetilde{G^*}^{000}_B \ne \widetilde{G^*}^{00}_B$.
\end{example}

\begin{proof}
By the discussion above, we know that $f$ is a quasi-character, so $\im(h)$ is finite. 
An easy investigation shows that in fact $\im(h)=\{-1,0,1\}$. It is also obvious that $h$ is $\emptyset$-definable in $\G$.

Assumptions (1) and (2) of Corollary \ref{cor:main2} are clearly satisfied. It is enough to prove that Assumption (5) holds, too. 

\begin{claim}
For every thick subset $P$ of $G$ there exists $g \in P$ such that for every $n \in \Z$, $f(g^n)=n$.
\end{claim}

\begin{proof}[Proof of the claim]
Put $b_i =a_2^ia_1^{-i}a_2a_1^i$ for $i=1,2,\dots$. Since $P$ is thick, there are $j>i$ such that $g:=b_ib_j^{-1} \in P$. We see that $g=a_2^ia_1^{-i}a_2a_1^{i-j}a_2^{-1}a_1^ja_2^{-j}$. So, $f(g)=(-i+1) +(i-j+1)+(j-1) =1$. 


The fact that $f(g^n)=n$ for $n \in \Z$ follows from an easy observation that while computing $g^n$ there are no cancellations involving the letter $a_1$.
\end{proof}
Assumption (5) of Corollary \ref{cor:main2} follows from the claim. \end{proof}
%
%


The following corollary has been proved in conversation with Anand Pillay.

\begin{corollary}\label{cor:grupa wolna} Let $F$ be the free group $\F_m$ (where $m\geq 2$) expanded by predicates for all subsets. Then ${F^*}^{000}_F={F^*}^{000}_\emptyset \ne {F^*}^{00}_\emptyset={F^*}^{00}_F$.
\end{corollary}

\begin{proof}
Denote $G=\F_m$. Let $\G$ be the expansion of $((\Z,+),(G,\cdot),f)$ by predicates for all subsets of $G$ and for all subsets of $\Z \times G$, where $f$ is the function from Example \ref{ex:new1}. Let $\G^*$ be a monster model such that the reduct $F^*$ of  $\G^*$ to the language of $F$ is a monster model extending $F$, and let $\Z^*$ and $G^*$ be the interpretations of $\Z$ and $G$ in $\G^*$. It is easy to see that ${G^*}^{000}_\emptyset={F^*}^{000}_\emptyset={F^*}^{000}_F$ and  ${G^*}^{00}_\emptyset={F^*}^{00}_\emptyset={F^*}^{00}_F$ 
Hence, it remains to show that ${G^*}^{000}_\emptyset \ne {G^*}^{00}_\emptyset$. 

Let $h$ and $\widetilde{G}$ be defined as in Example \ref{ex:new1}; $f^*$, $h^*$ and $\widetilde{G^*}$ denote the interpretations of $f$, $h$ and $\widetilde{G}$ in $\G^*$.
Since $h^*$ is split via the $\emptyset$-definable function $f^*$, we get that $\widetilde{G^*}$ is $\emptyset$-definably isomorphic with the product of groups $\Z^* \times G^*$. Since, from Example \ref{ex:new1}, we know that ${\widetilde{G^*}}^{000}_\emptyset \ne {\widetilde{G^*}}^{00}_\emptyset$, we conclude that either ${\Z^*}^{000}_\emptyset \ne  {\Z^*}^{00}_\emptyset$ or ${G^*}^{000}_\emptyset \ne {G^*}^{00}_\emptyset$. Suppose for a contradiction that ${G^*}^{000}_\emptyset = {G^*}^{00}_\emptyset$. Let $\sigma \colon G \to  \Z$ be an epimorphism. Since predicates for all subsets of $\Z \times G$ were added, $\sigma$ is  $\emptyset$-definable, so is its interpretation $\sigma^*$. Therefore, $\sigma^*\left[{G^*}^{000}_\emptyset \right]={\Z^*}^{000}_\emptyset \ne {\Z^*}^{00}_\emptyset = \sigma^*\left[{G^*}^{00}_\emptyset \right]$, a contradiction.
\end{proof}


Corollary \ref{cor:grupa wolna} and its  proof lead to the following question.

\begin{question}
Let $Z$ be the additive group of integers expanded by predicates for all subsets. Is it the case that ${Z^*}^{000}_Z ={Z^*}^{00}_Z$ or rather ${Z^*}^{000}_Z \ne {Z^*}^{00}_Z$. One can ask this question in the more general context of all abelian or nilpotent or solvable groups. Yet more generally, one can try to classify the groups for which the two connected components coincide (after expansion by predicates for all subsets). 
\end{question}

\subsubsection{Extensions of surface groups}

Let $\Gamma_g$ be the fundamental group of a closed oriented surface of genus $g \geq 2$. The standard presentation of $\Gamma_g$ is
\[\Gamma_g=\left\langle a_1,\dots,a_g,b_1,\dots,b_g : \prod_{i=1}^g [a_i,b_i]=1 \right\rangle,\]
where $[a,b]=a^{-1}b^{-1}ab$. Let $F=\{ a_1,\dots,a_g,b_1,\dots,b_g\}$.

In \cite [Page 152]{Gr}, the notion of an \emph{admissible} word is explained, and it is recalled from another paper that any element $x \in \Gamma_g$ can be uniquely written as an admissible word; in such a situation, we say that $x$ is in \emph{standard form}. Roughly speaking, admissible words are the words corresponding to the `simplest' paths in the Cayley graph of $\Gamma_g$.

Let $W$ be a freely reduced word in the alphabet $F \cup F^{-1}$. For $x \in \Gamma_g$ we define $h_W(x)$ as the number of occurrences of $W$ as a subword in the admissible word representing $x$. Define $f_W: \Gamma_g \to \Z$ by
\[f_W(x)=h_W(x)-h_{W^{-1}}(x)-h_W(x^{-1})+h_{W^{-1}}(x^{-1}).\]
The next fact is \cite[Proposition 5.14]{Gr}

\begin{fact} For every freely reduced word $W$ the function $f_W$ is a quasi-character. Moreover, $f_{W}(x^{-1})=-f_W(x)$ for all $x \in \Gamma_g$.
\end{fact}

Similarly to the case of free groups, one can also define quasi-characters $e_W(x)$, using so-called \emph{admissible cyclic} words; in the case when $W$ is without self-overlapping, $e_W$ is a pseudo-character (see \cite[Page 156]{Gr}).

Arguing as  in the proof of Example \ref{ex:new1}, we get the following example.

\begin{example}\label{ex:new2}
Let $G=\Gamma_g$ (where $g \geq 2$) and $A=(\Z,+)$. Take $f=f_{a_1^2}$ or $f=e_{a_1^2}$ (i.e., we take $W:=a_1^2$ in the discussion above). Define $h: \F_m \times \F_m \to \Z$ by $h(x,y)=f(x)+f(y)-f(xy)$ for $x,y \in \Gamma_g$. 
Then $f$ is a quasi-character, so $h$ has finite image. Let $B=\{ 1\}$. Take $\G=((\Z,+),(\Gamma_g,\cdot),f)$, and let $\G^*=((\Z^*,+),(\Gamma_g^*,\cdot),f^*) \succ \G$ be a monster model. Let $\widetilde{G}$ be the central extension of $G$ by $\Z$ defined by means of $h$. Put $A^*_1=\bigcap_{n \in \N \setminus \{ 0 \}} n\Z^*$. Then the assumptions of Corollary \ref{cor:main2} are satisfied, so $\widetilde{G^*}^{000}_B \ne \widetilde{G^*}^{00}_B$.
\end{example}

\begin{proof}
It is enough to prove the following
\begin{claim}
For every thick subset $P$ of $G$ there exists $x \in P$ such that for every $n \in \Z$, $f(x^n)=2n$.
\end{claim}
\begin{proof}[Proof of the claim.]
We consider the elements $b_i$ defined as in the proof of the claim in Example \ref{ex:new1}, and, in order to
to avoid a clash of notation, we denote by $x$ the element $g$ defined in that proof. Recall that $x=a_2^ia_1^{-i}a_2a_1^{i-j}a_2^{-1}a_1^ja_2^{-j}$ and $x^{-1}=a_2^ja_1^{-j}a_2a_1^{j-i}a_2^{-1}a_1^ia_2^{-i}$. It is clear from the definition of admissible words \cite[Page 152]{Gr} that both these elements are written in standard form. Thus, $f(x)=2(h_{a_1^2}(x)-h_{a_1^{-2}}(x))=2$. Since the powers of $x$ are also in standard form (after obvious cancellations of letters $a_2$), we conclude that $f(x^n)=2n$ for $n \in \Z$, which finishes the proof.
\end{proof}
The justification of Example \ref{ex:new2} is completed. \end{proof}
%
%
Using the same argument as in the proof of Corollary \ref{cor:grupa wolna}, we get the following corollary. (Note that there is an epimorphism from $\Gamma_g$ to $\Z$, e.g. sending $a_1$ to $1$ and the other generators to $0$).

\begin{corollary} Let $\Gamma$ be the surface group $\Gamma_g$ (where $g \geq 2$) expanded by predicates for all subsets. Then ${\Gamma^*}^{000}_F={\Gamma^*}^{000}_\emptyset \ne {\Gamma^*}^{00}_\emptyset={\Gamma^*}^{00}_F$.
\end{corollary}

\subsubsection{Extensions of $\Z_m * \Z_n$}

Let $G=\Z_m \times \Z_n = \langle a,b: a^m=b^n=e\rangle$, where $m,n \geq 2$. 
Consider the set \[F=\{a_1,\dots,a_{m-1}, b_1,\dots,b_{n-1}\}\] of generators of $G$, where $a_i=a^i$ for $1\leq i \leq m-1$ and $b_j=b^j$ for $1\leq j \leq n-1$. Every element $g \in G$ can be uniquely written in \emph{normal form}, i.e., as a word $W$ in the alphabet $F$ in which each symbol $a_i$ is followed by some symbol $b_j$ and each symbol $b_j$ is followed by some symbol $a_i$. We say that such a word $W$ is \emph{reduced}. If $W=a_{i_1}b_{j_1}\dots a_{i_k}b_{j_k}$ is written in normal form, we define $W^-=b_{n-j_k}a_{m-i_k}\dots b_{n-j_1}a_{m-i_1}$ (i.e., $W^-$ is the normal form of the element $W^{-1}$).

Let $W$ be a reduced word. Define $f_W \colon G \to \Z$ by
\[f_W(g)= \mbox{number of occurrences of $W$ in $g$}\; - \; \mbox{number of occurrences of $W^-$ in $g$},\]
where by an `occurrence of $W$ in $g$' we mean any occurrence of $W$ as a subword of the normal form of $g$. 

For an element $g \in G$, denote by $\overline{g}$ the {\em reduced cyclic} word corresponding to $g$ (reduced still means without subwords $a_ib_j$ or $b_ja_i$).
Define $e_W \colon G \to \Z$ by
\[e_W(g)= \mbox{number of occurrences of $W$ in $\overline{g}$}\; - \; \mbox{number of occurrences of $W^-$ in $\overline{g}$}.\]

\begin{fact}
For any reduced word $W$, $f_W$ and $e_W$ are quasi-characters satisfying $f_W(g^{-1})=-f_W(g)$ and $e_W(g^{-1})=-e_W(g)$. If $W$ is without self-overlapping, then $e_W$ is a pseudo-character.
\end{fact}

\begin{example}\label{ex:new3}
Let $G=\Z_m * \Z_n$, where $m,n > 3$, and  $A=(\Z,+)$. Take $f=f_{ab}$ or $f=e_{ab}$ (i.e., we take $W:=ab$ in the discussion above). Define $h: G \times G \to \Z$ by $h(x,y)=f(x)+f(y)-f(xy)$ for $x,y \in G$. 
Then $f$ is a quasi-character, so $h$ has finite image. Let $B=\{ 1\}$. Take $\G=((\Z,+),(G,\cdot),f)$, and let $\G^*=((\Z^*,+),(G^*,\cdot),f^*) \succ \G$ be a monster model. Let $\widetilde{G}$ be the central extension of $G$ by $\Z$ defined by means of $h$. Put $A^*_1=\bigcap_{n \in \N \setminus \{ 0 \}} n\Z^*$. Then the assumptions of Corollary \ref{cor:main2} are satisfied, so $\widetilde{G^*}^{000}_B \ne \widetilde{G^*}^{00}_B$.

\end{example}

\begin{proof}
As in the proof of Example \ref{ex:new1}, it is enough to prove the following claim.
\begin{claim}
For every thick subset $P$ of $G$ there exists $g \in P$ such that for every $k \in \Z$, $f(g^k)=2k$.
\end{claim}
\begin{proof}[Proof of the claim.]
Put $b_i=(ba^2)^ia(ab)^{-i}b(ab)^i$ for $i=1,2\dots$. Since $P$ is thick, there are $j>i$ such that $g:=b_ib_j^{-1} \in P$. We see that 
\[\begin{array}{lrl}
g &= &\underbrace{ba^2\dots ba^2}_{i-1}ba^3 \underbrace{b^{n-1}a^{m-1}\dots b^{n-1}a^{m-1}}_{i-1} b^{n-1} a^{m-2} \underbrace{b^{n-1}a^{m-1}\dots b^{n-1}a^{m-1}}_{j-i-1} \\
& \cdot& b^{n-1} \underbrace{ab \dots ab}_{j} a^{m-3}b^{n-1}\underbrace{a^{m-2}b^{n-1} \dots a^{m-2}b^{n-1}}_{j-1}
\end{array}\]
is written in normal form. Thus, $f(g)=-(i-1) - (j-i-1)+j=2$. From this, we easily conclude that $f(g^k)=2k$ for all $k \in \Z$.
\end{proof}
The justification of Example \ref{ex:new3} is completed.
\end{proof}
%
%

\subsubsection{An extension of $\SL_2(\Z)$: the braid group on three strands}

It has been an open question (at least for us) for a while whether the extension of $\SL_2(\Z)$ by $\Z$ defined by means of the 2-cocycle defining the universal cover of $\SL_2(\R)$ is an example where the two connected components are different. Here, we will prove that this is true. In order to do that, we will apply some results from \cite{asai} and Corollary \ref{cor:main2}; at the end of the paper, we will explain that none of the examples from Subsection \ref{subsection:4.2} could be obtained by methods from \cite{conv_pillay}.

Throughout this part of the paper, we will often be using (without mentioning) the formula for the 2-cocycle $h$ defining the universal cover of $\SL_2(\R)$, which was found in \cite{asai} and which has already been recalled in Example \ref{ex:PC1}. 


For a matrix 
$M=\left( 
\begin{array}{cc}
a & b \\
c & d
\end{array}\right)$
we define $a(M)=a$, $b(M)=b$, $c(M)=c$, $d(M)=d$ and $\Tr(M)=a+d$.


\begin{lemma}\label{lem:existence of A}
Consider an arbitrary first order expansion $\G$ of the 2-sorted structure \[((\Z,+),(\Z,+,\cdot)).\] 
Let $B$ be a set of parameters containing number 1 from the first sort. Then $G:=\SL_2(\Z)$ and $h$ are $B$-definable in $\G$ ($h$ is treated as a function from the set $G\times G$ which is $\emptyset$-definable in the sort $(\Z,+,\cdot)$ to the sort $(\Z,+)$). Let $\G^* \succ \G$ be a monster model.
%
%
Then, there exists a matrix $M \in {G^*}^{00}_B$ such that $h^*(M,M)=1$ and $c(M)>0$. In particular, $c(M^n) \ne 0$ for all $n \in \Z \setminus \{ 0 \}$, and $c(M^{2^n})<0$ for all $n \in \N \setminus \{ 0\}$.
\end{lemma}

\begin{proof}
We are going to use the following observation: for any matrix
$M=\left( 
\begin{array}{cc}
a & b \\
c & d
\end{array}\right)$ with coefficients in $\Z^*$,
since $c(M^2)=ca+cd=c(M)\Tr(M)$, the formula for $h$ implies that if $c(M):=c>0$ and $\Tr(M):=a+d<0$, then $h^*(M,M)=1$. More precisely, we will prove that such a matrix $M$ can be chosen in ${G^*}^{00}_B=\SL_2(\Z^*)^{00}_B$. 

Consider a sequence of matrices $(M_i)_{i<\kappa}$ in $G^*$ (here $\kappa$ is the degree of saturation of $\G^*$), where
$M_i:=\left( 
\begin{array}{cc}
a_i & a_id_i-1 \\
1 & d_i
\end{array}\right)$
for some sequences $(a_i)$ and $(b_i)$ of elements of $\Z^*$ satisfying $a_j-a_i\geq2$ and $b_j-b_i\geq 2$ for all $i<j<\kappa$.
As ${G^*}^{00}_B$ is of bounded index, there are $j>i$ such that $M:=M_iM_j^{-1} \in {G^*}^{00}_B$. A computation yields 
\[M=
\left( 
\begin{array}{cc}
a_i(d_j-d_i) +1 & a_ia_j(d_i-d_j) + a_i-a_j \\
d_j-d_i & a_j(d_i-d_j) +1
\end{array}\right).\]
Thus, $c(M)=d_j-d_i>0$ and $Tr(M)=(a_i-a_j)(d_j-d_i)+2<0$. So, we conclude that $c(M^2)<0$ and $h^*(M,M)=1$.

Now, we show that $c(M^n) \ne 0$ for all $n \in \Z \setminus \{0\}$. Since the standard congruence subgroups 
\[\Gamma(n):=\left\{
\left( \begin{array}{cc}
a & b \\
c & d
\end{array}\right)
\in \SL_2(\Z) : 
\left( \begin{array}{cc}
a & b \\
c & d
\end{array}\right)
\equiv
\left(\begin{array}{cc}
1 & 0 \\
0 & 1
\end{array}\right)
\mod n \right\}\]
of $\SL_2(\Z)$ are $\emptyset$-definable in $\G$ and of finite index, we see that ${G^*}^{00}_B < \Gamma(n)^*$ for every $n \in \N \setminus \{ 0 \}$. So, we get
that $b(M),c(M) \in {\Z^*}^0$ and $a(M),d(M) \in 1+{\Z^*}^0$, where ${\Z^*}^0:=\bigcap_{n \in \N \setminus \{ 0 \}} n\Z^*$.
By an easy induction, one shows that $a(M^n) \in 1+{\Z^*}^0$, $b(M^n) \in {\Z^*}^0$, $c(M^n) \in c(M)(n+ {\Z^*}^0)\subseteq {\Z^*}^0 \setminus \{ 0 \}$ and $d(M^n) \in 1 +{Z^*}^0$ for all $n \in \N \setminus \{ 0 \}$; in particular, $c(M^n) \ne 0$, and the formula for the inverse of a matrix in $\SL_2(\Z^*)$ yields this conclusion for all $n \in \Z \setminus \{ 0 \}$. 

It remains to show that  $c\left(M^{2^n}\right)<0$ for all $n \in \N \setminus \{ 0\}$. The proof is by induction on $n$. We know that this is true for $n=1$. Suppose it holds for some $n$. One easily checks that $\Tr(M^{2^n})= \Tr(M^{2^{n-1}})^2-2$. Thus, since $a(M^{2^{n-1}}), d(M^{2^{n-1}}) \in 1 + {\Z^*}^0$, we get that  $\Tr(M^{2^n})>0$. Using this together with the induction hypothesis that $c(M^{2^n})<0$, we get that $c(M^{2^{n+1}}) = c(M^{2^n})\Tr\left(M^{2^n}\right)<0$.
\end{proof}

\cite[Theorem 3]{asai} tells us that $h_{| \SL_2(\Z) \times \SL_2(\Z)} \colon \SL_2(\Z) \times \SL_2(\Z) \to \R$ is split via a unique real-valued function $f \colon \SL_2(\Z) \to \R$. Moreover, this function $f$ is rational-valued, and $12\im(f) \subseteq \Z$. 


\begin{remark}\label{rem:integer-valued}
$f\left[\left[\SL_2(\Z),\SL_2(\Z)\right]\right] \subseteq \Z$.
\end{remark}
 
\begin{proof}
It is clear that $f(I)=h(I,I)=0$. All the time in this proof, we use the fact that $\im(h) \subseteq \Z$.
Notice that if $f(A) \in \Z$, then $f(A^{-1}) \in \Z$. Indeed, $f(A^{-1})= h(A,A^{-1})-f(A) +f(I) \in \Z$. By an easy induction, we get that if $f(A_1),\dots, f(A_n) \in \Z$, then for any $\xi_1,\dots,\xi_n \in \{1,-1\}$, one has $f(A_1^{\xi_1} \cdot \ldots \cdot A_n^{\xi_n}) \in \Z$. So, it remains to check that for any $A,B \in \SL_2(\Z)$, $f([A,B]) \in \Z$. Since $\im(h) \subseteq \Z$, we have $f([A,B])=f(A^{-1}B^{-1}AB) = f(A^{-1}B^{-1}) + f(AB) -h(A^{-1}B^{-1},AB) = f(A^{-1})+f(A) +f(B^{-1})+f(B) -h(A^{-1},B^{-1})-h(A,B) -h(A^{-1}B^{-1},AB)\in  f(A^{-1})+f(A) +f(B^{-1})+f(B) +\Z = f(I)+h(A^{-1},A) + f(I) +h(B^{-1},B) +\Z =\Z$.
\end{proof}

It is a classical fact that the congruence subgroup $\Gamma(12)$ has finite index in $\SL_2(\Z)$ and is contained in $[\SL_2(\Z),\SL_2(\Z)]$. Thus, $[\SL_2(\Z),\SL_2(\Z)]$ is $\emptyset$-definable in $(\Z,+,\cdot)$. From now on, instead of $h_{|\SL_2(\Z) \times \SL_2(\Z)}$ we will just write $h$.


\begin{example}\label{ex:new4}
Let $G=\SL_2(\Z)$, $H=[\SL_2(\Z),\SL_2(\Z)]$ and $\G=((\Z,+), (\Z,+,\cdot), f_{|H})$ [more generally, $\G$ is any expansion of this 2-sorted structure]; here, $f_{|H}$ is treated as a function from the group $H$ $\emptyset$-definable in in the sort $(\Z,+,\cdot)$ to the sort $(\Z,+)$ (see Remark \ref{rem:integer-valued}). As usual, $\G^* \succ \G$ denotes a monster model. Let $\widetilde{G}$ be the central extension of $G$ by $\Z$ defined by means of $h$. Put $A=(\Z,+)$, $A_1^*=\bigcap_{n \in \N \setminus \{ 0 \}} n\Z^*$ 
and $B=\{ 1\}$. Then Assumptions (1), (2) and (4) (so also (3)) of Corollary \ref{cor:main2} are satisfied, and so $\widetilde{G^*}^{000}_B \ne \widetilde{G^*}^{00}_B$.  
\end{example}

\begin{proof}
Assumptions (1) and (2) are clearly satisfied. It remain to show that (4) also holds. For this it is enough to prove the following claim. Instead of ${f_{|H}}^*$, we will write $f^*$ (note that the arguments of $f^*$ considered below are always from ${G^*}^{00}_B \subseteq H^*$).

\begin{claim}
Let $M\in {G^*}^{00}_B$ be a matrix provided by Lemma \ref{lem:existence of A}, i.e., it satisfies $h^*(M,M)=1$, $c(M)>0$, $c(M^n) \ne 0$ for all $n \in \Z \setminus \{ 0 \}$, and  $c(M^{2^n})<0$ for all $n \in \N \setminus \{ 0\}$. Then $2f^*(M) -1 \notin A^*_1$ and $\Z \cdot (2f^*(M)-1) \subseteq f^*\left[{G^*}^{00}_B\right]$.
\end{claim}

\begin{proof}[Proof of the claim.]
It is clear that $2f^*(M) -1 \notin A^*_1$. In order to see that $\Z \cdot (2f^*(M)-1) \subseteq f^*[{G^*}^{00}_B]$, we will prove by induction the following properties:
\begin{enumerate}
\item[(i)] $f^*(M^{2n})=2nf^*(M)-n$ and $c\left(M^{2n}\right)<0$ for all $n \in \N \setminus \{ 0 \}$,
\item[(ii)] $f^*(M^{2n-1})=(2n-1)f^*(M) - (n-1)$ and $c(M^{2n-1})>0$ for all $n \in \N \setminus \{ 0 \}$.
\end{enumerate}
{\bf Base step $(n=1)$.} (ii) is clear from the choice of $M$. We also know that $c(M^2)<0$. Finally, we have $1=h^*(M,M)=2f^*(M)-f^*(M^2)$, so $f^*(M^2)=2f^*(M) -1$.\\[1mm]
{\bf Induction step.} Suppose (i) and (ii) hold for all $n \in \{1,\dots,m\}$ for some $m\geq 1$. Choose the biggest natural number $k$ such that $m\geq 2^k$. Put $l=2^k$. Then $m<2l \leq 2m$. 

Since $c(M^{2l})<0$ and  $c(M^{4l})<0$, we have
$0=h^*(M^{2l},M^{2l})=2f^*(M^{2l})-f^*(M^{4l})$, so, by induction hypothesis, we get
\begin{equation}
f^*\left(M^{4l}\right)=4lf^*(M)-2l.\tag{$\dag$} \label{eq:4l}
\end{equation}

As, by induction hypothesis,  $c(M)>0$ and $c(M^{2m})<0$, we have $0=h^*(M,M^{2m})=f^*(M) + f^*(M^{2m})-f^*\left(M^{2m+1}\right)$, and so, by induction hypothesis, 
\[f^*\left(M^{2m+1}\right)= (2m+1)f^*(M)-m,\] 
i.e., the first equality of (ii) holds for $n=m+1$.
Thus, using induction hypothesis and (\ref{eq:4l}), we get
$h^*\left(M^{2(2l-m)-1},M^{2m+1}\right)=f^*\left(M^{2(2l-m)-1}\right)+ f^*\left(M^{2m+1}\right) -f^*\left(M^{4l}\right)=(2(2l-m)-1)f^*(M) -((2l-m)-1) +(2m+1)f^*(M) - m - (4lf^*(M) -2l)=1$. This implies that
\[c(M^{2m+1})>0,\]
which completes the proof of (ii) for $n=m+1$.

If $m=2l-1$, then (i) for $n=m+1$ follows from (\ref{eq:4l}) and the fact that $c(M^{4l})<0$. So, consider the case $2l-m>1$.
Since,  by induction hypothesis, $c(M^{2(2l-m-1)})<0$ and $c(M^{4l})<0$, we have $0=h^*(M^{2(2l-m-1)},M^{2m+2}) = f^*(M^{2(2l-m-1)})+f^*(M^{2m+2}) -f^*(M^{4l})$. Hence, by induction hypothesis and (\ref{eq:4l}), we get
\[f^*(M^{2m+2})= (2m+2)f^*(M) -(m+1),\]
i.e., the first equality in (i) holds for $n=m+1$.
We also have 
$h^*(M,M^{2m+1})= f^*(M)+f^*(M^{2m+1}) - f^*(M^{2m+2})= f^*(M)+(2m+1)f^*(M) - m - ((2m+2)f^*(M) - (m+1))=1$. This implies that \[c\left(M^{2m+2}\right)<0,\] which completes the proof of (i) for $n=m+1$.
\end{proof}

As was noted, the claim finishes our justification of Example \ref{ex:new4}.
\end{proof}

Using the fact that $h$ is split via $f$ and $\im(12f) \subseteq \Z$, we can rewrite the proof of Example \ref{ex:new4} to get the following example.
\begin{example}\label{ex:new5}
Replace $h$, $f$, $B$ and $\G$ in Example \ref{ex:new4} by $12h$, $12f$, $12B=\{ 12\}$, and $\G_{12}:=((\Z,+), (\Z,+, \cdot), 12f)$; in particular, $\widetilde{G}$ is the central extension of $\SL_2(\Z)$ by $\Z$ defined by means of the 2-cocycle $12h$. Then Assumptions (1), (2) and (4) (so also (3)) of Corollary \ref{cor:main2} are satisfied, and so $\widetilde{G^*}^{000}_B \ne \widetilde{G^*}^{00}_B$. 
\end{example}

Now, we replace the sort $(\Z,+,\cdot)$ by the pure group $\SL_2(\Z)$.

\begin{example}
Let $G=\SL_2(\Z)$, $H=\left[\SL_2(\Z),\SL_2(\Z)\right]$ and $\G=((\Z,+), \SL_2(\Z), H, 12f)$, where $(\Z,+)$ and $\SL_2(\Z)$ are two sorts, $H$ is a predicate in the sort $\SL_2(\Z)$, and $12f$ is treated as a function from the sort $\SL_2(\Z)$ to the sort $(\Z,+)$. Then $h$ is $\emptyset$-definable in $\G$. As usual, $\G^* \succ \G$ denotes a monster model. Let $\widetilde{G}$ be the central extension of $G$ by $\Z$ defined by means of $h$, 
and $B:=\{1\}$. Put $A=(\Z,+)$ and $A_1^*=\bigcap_{n \in \N \setminus \{ 0 \}} n\Z^*$. Then Assumptions (1), (2) and (4) (so also (3)) of Corollary \ref{cor:main2} are satisfied, and so $\widetilde{G^*}^{000}_B \ne \widetilde{G^*}^{00}_B$.  
\end{example}

\begin{proof}
Since $12h$ is $\emptyset$-definable (using splitness via $12f$), we get that $h$ is $\emptyset$-definable. It is clear that Assumptions (1) and (2) of Corollary \ref{cor:main2} are satisfied (as $f_{|H}$ is $\emptyset$-definable in $\G$). Since the structure $\G$ is $\emptyset$-interpretable in one of the structures $\G$ considered in Example \ref{ex:new4}, 
we have that ${G^*}^{00}_B$ computed in the current example contains ${G^*}^{00}_B$ computed in Example \ref{ex:new4} (working in an appropriate structure). Thus, the fact that  Assumption (4) of Corollary \ref{cor:main2} holds in Example \ref{ex:new4} implies that it holds in the current example.
\end{proof}

\begin{example}
Let $G=\SL_2(\Z)$, $H=\left[\SL_2(\Z),\SL_2(\Z)\right]$ and $\G=((\Z,+), \SL_2(\Z), 12f)$, where $(\Z,+)$ and $\SL_2(\Z)$ are two sorts, and $12f$ is treated as a function from the sort $\SL_2(\Z)$ to the sort $(\Z,+)$. Then $h$ is $\emptyset$-definable in $\G$. Replace $h$ by $12h$, i.e., $\widetilde{G}$ is the extension of $G$ by $\Z$ defined by means of $12h$, 
and let $B:=\{12\}$. As usual, $\G^* \succ \G$ denotes a monster model, $A:=(\Z,+)$ and $A_1^*:=\bigcap_{n \in \N \setminus \{ 0 \}} n\Z^*$. Then Assumptions (1), (2) and (4) (so also (3)) of Corollary \ref{cor:main2} are satisfied, and so $\widetilde{G^*}^{000}_B \ne \widetilde{G^*}^{00}_B$.
\end{example}

\begin{question}
Let $\widetilde{G}$ be the extension of $\SL_2(\Z)$ by $\Z$ defined by means of $h$ [or $12h$]. Let $\G$ be the pure group structure on $\widetilde{G}$, and $\G^* \succ \G$ a monster model. Is it true that $\widetilde{G^*}^{000}_B \ne \widetilde{G^*}^{00}_B$ for some small set of parameters $B$?
\end{question}

Let us note the following consequence of Corollary \ref{cor:grupa wolna} and the fact that the commutator subgroup $[\SL_2(\Z),\SL_2(\Z)]$ is a free group in two generators (see \cite[Page 117, Theorem 6.9.]{aig}) and it has finite index in $\SL_2(\Z)$.

\begin{corollary} Let $S$ be the group $\SL_2(\Z)$ expanded by predicates for all subsets. Then ${S^*}^{000}_S={S^*}^{000}_\emptyset \ne {S^*}^{00}_\emptyset={S^*}^{00}_S$.
\end{corollary}

Now, we briefly explain why none of the examples obtained in Subsection \ref{subsection:4.2} could be obtained by the methods from \cite{conv_pillay}. The proof from \cite{conv_pillay} uses commutator subgroups; more precisely, it uses the fact that $\widetilde{\SL_2(\R)}$ is perfect, which is not true for $\widetilde{G}$ in our examples (because of splitness and Remark \ref{rem:nice and easy}(1)). Even if one tries to do more delicate arguments in the spirit of \cite{conv_pillay}, one would need to know that $\left[\widetilde{G^*}^{00}_B /A^*_1, \widetilde{G^*}^{00}_B /A^*_1\right] \cap A^*/A^*_1$ is non-trivial, but in our examples it is trivial by the splitness of $\overline{h^*}_{|{G^*}^{00}_B \times {G^*}^{00}_B} \colon {G^*}^{00}_B \times {G^*}^{00}_B \to A^*/A^*_1$ and Remark \ref{rem:nice and easy}(1). 

We finish with a remark that applying Corollary \ref{prop:more examples in 4.2} to situations like in Examples \ref{ex:new1}, \ref{ex:new2} or \ref{ex:new3}, one can produce wider classes examples where the two connected components differ.


\end{document}